\documentclass[11pt,a4paper]{amsart}

\usepackage[T1]{fontenc}
\usepackage[utf8]{inputenc}
\usepackage{lmodern}

\usepackage[english]{babel}
\usepackage{amsmath,amsfonts,amssymb,amsthm}

\usepackage[pdftex]{graphicx}

\usepackage[bottom=3cm, top=3cm, left=3.5cm, right=3.5cm]{geometry}

\usepackage[foot]{amsaddr}
\usepackage[alphabetic, abbrev]{amsrefs}

\usepackage{mathtools}
\mathtoolsset{showonlyrefs,showmanualtags}	

\usepackage{bm}	
\DeclareMathAlphabet{\mathbbold}{U}{bbold}{m}{n}	

\theoremstyle{plain}
\newtheorem{theorem}{Theorem}
\newtheorem*{theorem*}{Theorem}
\newtheorem{proposition}[theorem]{Proposition}
\newtheorem{corollary}[theorem]{Corollary}
\newtheorem{lemma}[theorem]{Lemma}
\newtheorem*{lemma*}{Lemma}

\theoremstyle{definition}
\newtheorem{definition}[theorem]{Definition}

\theoremstyle{remark}
\newtheorem{remark}[theorem]{Remark}

\usepackage[bookmarksdepth=3]{hyperref}
\usepackage{xcolor}
\hypersetup{
    colorlinks,
    linkcolor={red!50!black},
    citecolor={blue!50!black},
    urlcolor={blue!80!black}
}

\usepackage{todonotes}

\usepackage{enumitem}
\setlist[itemize]{leftmargin=1cm}

\newcommand{\R}{\mathbb{R}}

\newcommand{\distr}{\mathcal{D}}
\newcommand{\eps}{\varepsilon}

\newcommand{\hor}{\mathcal{H}}

\newcommand{\Tor}{\mathrm{Tor}}
\newcommand{\vol}{\mathrm{vol}}

\newcommand{\Popp}{\nu}
\newcommand{\diff}{\mathrm{d}}

\let\oldtocsection=\tocsection 
\let\oldtocsubsection=\tocsubsection

\renewcommand{\tocsection}[2]{\hspace{0em}\oldtocsection{#1}{#2}}
\renewcommand{\tocsubsection}[2]{\hspace{1em}\oldtocsubsection{#1}{#2}}

\author[Davide Barilari]{Davide Barilari$^*$}
\address{$^*$ Dipartimento di Matematica ``Tullio Levi-Civita'', Universit\`a degli Studi di Padova}
\email{\href{mailto:davide.barilari@unipd.it}{davide.barilari@unipd.it}}

\author[Tania Bossio]{Tania Bossio$^*$}
\email{\href{mailto:tania.bossio@phd.unipd.it}{tania.bossio@phd.unipd.it}}

\title[Tubes and Steiner in 3D contact SR geometry]{Steiner and tube formulae in 3D contact sub-Riemannian geometry}

\subjclass[2010]{53C17, 49J15}

\date{\today}
\begin{document}

\begin{abstract}
We prove a Steiner formula for regular surfaces with no characteristic points in 3D contact sub-Riemannian manifolds endowed with an arbitrary smooth volume. The formula we obtain, which is equivalent to a half-tube formula, is of local nature. It can thus be applied to any surface in a region not containing characteristic points.  

We provide a geometrical interpretation of the coefficients appearing in the expansion, and compute them on some relevant examples in three-dimensional sub-Riemannian  model spaces.  These results generalize those obtained in \cite{balogh15} and \cite{ritore21} for the Heisenberg group. 

\end{abstract}

\maketitle

\setcounter{tocdepth}{1}
\tableofcontents

\section{Introduction}
\label{s:introduction}

The celebrated Steiner's formula states that the Lebesgue volume of the $\varepsilon$-neigh\-bor\-hood $\Omega_\eps$ of a bounded regular domain $\Omega\subset\R^n$ can be expressed as a polynomial of degree $n$ in $\eps$
\begin{equation}\label{eq:steiner}
\vol(\Omega_{\eps})=\vol\left(\Omega\right)+\sum_{k=1}^na_k(\Omega)\eps^k,
\end{equation} 
where the $a_{k}=a_k(\Omega)$ are coefficients known as Minkowski's quermassintegrals.

The Steiner formula is naturally related with the so-called tube formula, or, to be precise, the half-tube formula.
Given a compact hypersurface $S$ embedded in $\R^n$, one defines the tube of size $\eps$ around $S$ as the set of points at distance at most $\varepsilon$ from $S$. 
If the hypersurface $S$ bounds a regular domain $\Omega$, one might define the half-tube $T_\eps(\Omega)$ of size $\eps$ by considering only the ``external tube'', i.e., points of the tube that do not belong to the domain $\Omega$. 
With this convention, it is clear that $T_\eps(\Omega)=\Omega_{\varepsilon}\setminus \Omega$, and the half-tube formula is  
\begin{equation}
\vol(T_{\eps}(\Omega))=\sum_{k=1}^{n}a_{k}(\Omega)\varepsilon^{k}. 
\end{equation}

In the case of a compact surface $S=\partial \Omega$ bounding a regular domain $\Omega $ in $\R^3$ the coefficients $a_{1},a_{2},a_{3}$ are computed explicitly as follows
\begin{equation}\label{eq:steiner2}
\vol(T_{\eps})=\eps \int_{S}  dA- \eps^{2} \int_{S} H dA + \frac{\eps^{3}}{3} \int_{S} K dA,
\end{equation} 
where $H$ is the mean curvature of $S$, computed with respect to the orientation of the surface given by the outward pointing normal of $\Omega$, and $K$ is the Gaussian curvature of $S$, while $dA$ denotes the surface measure.

The original work of  Steiner proves  formula \eqref{eq:steiner} for convex sets in the Euclidean spaces of dimension $2$ and $3$ \cite{steiner}. 
In the same years, Weyl proved the tube formula for smooth submanifolds of the Euclidean space \cite{Weyl39}. 
We refer the  reader to the monograph of A.\ Gray \cite{graybook} for an exhaustive overview of this subject. 

\medskip

The goal of this paper is to investigate the validity of Steiner like (or half-tube like) formulas in sub-Riemannian geometry. More precisely, we consider the case of half-tubes built over surfaces in three-dimensional sub-Riemannian contact manifolds. We provide a geometrical interpretation of the coefficients appearing in the expansion, generalizing the results obtained in \cite{balogh15,ritore21} for surfaces which are embedded in the Heisenberg group. 

A three-dimensional contact sub-Riemannian manifold is a triple $(M,\distr,g)$ where $M$ is a three-dimensional manifold, $\distr$ a bracket generating two-dimensional distribution in the tangent bundle, locally obtained as the kernel of a one-dimensional smooth form $\omega$ such that $\omega \wedge \mathrm{d}\omega\neq 0 $, and $g$ a smooth metric defined on the distribution such that the two-dimensional induced volume form $\mathrm{vol}_g$ coincides with the contact two-dimensional form $\mathrm{d}\omega|_{\distr}$. 
The three-dimensional form $\omega\wedge \mathrm{d}\omega$ is then a canonical volume form associated with the contact structure and the volume measure $\nu$ is called the \emph{Popp's volume} (for more details see \cite{BR15}). 

In what follows we assume a smooth measure $\mu$ on the manifold $M$ is given, which in general can be written as $\mu=h \nu$ for a smooth positive function $h$. 
  
 \subsection{Surfaces in 3D contact manifolds} 
Let $M$ be a three-dimensional contact sub-Riemannian manifold and let $S$ be a surface with no characteristic points (namely $\distr_{x}\neq T_{x}S$ for every $x\in S$) and that bounds a closed regular domain $\Omega\subset M$. 
We study the measure $\mu$ of the sub-Rieman\-nian half-tube built over $S$, that is defined through the sub-Rieman\-nian distance.

Given $S$ a hypersurface with no characteristic points, the sub-Rieman\-nian normal $N$ to the surface $S$ is a vector field that generates the unique direction in the distribution orthogonal to $T_{x}S$ (which is well-defined thanks to the condition $\distr_{x}\neq T_{x}S$).
If the surface $S$ is locally defined as $S=f^{-1}(0)$ where $f:M\to\R$ is a smooth function with $\diff f\neq 0$, the vector field $N$ on $S$ is parallel to the horizontal gradient $\nabla_Hf$ and of unit norm with respect the metric $g$. 
Therefore, $N$ is unique up to a sign.
Recall that the horizontal gradient $\nabla_Hf$ is the unique vector field tangent to the distribution $\distr$ that satisfies the identity
\begin{equation}
\diff f(X)=g(\nabla_Hf,X)
\end{equation} 
for every vector field $X$ tangent to $\distr$. 
Given an orthonormal basis $X_1$, $X_2$ on $\distr$, it holds that $\nabla_Hf=\left(X_1f\right)X_1 + \left(X_2f\right)X_2$.  We also denote by $X_{0}$ the Reeb vector field associated with the contact form which, we recall, is transversal to $\distr$.

Let us denote with $dV_\mu$ the 3-dimensional smooth volume form on $M$ associated with $\mu$. Namely, $dV_\mu$ is such that $\mu(B)=\int_BdV_\mu$ for every $B\subset M$ measurable. 
Furthermore, we can define a a smooth measure on the surface $S$ associated with $\mu$, that is the \emph{induced sub-Rieman\-nian area} $\sigma_\mu$: 
for a measurable set $U\subset S$, we set
\begin{eqnarray}
\label{eq:dA}
\sigma_\mu (U)=\int_UdA_{\mu}, & \mbox{where} & dA_{\mu}:=\iota_{N}dV_\mu
\end{eqnarray}
is the 2-dimensional volume form on $S$ obtained by restricting to $N$ the volume $dV_\mu$. 

The 3-dimensional volume form $dV_\mu$ and the sub-Rieman\-nian area form induced on a hypersurface are related by the following sub-Rieman\-nian version of the coarea formula, which plays a crucial role in the proof of our result.
Given a measurable function $f:M\to\R$ and a smooth map $\Phi:M\to\R$ such that $\nabla_H\Phi\neq 0$ on $M$, it holds that
\begin{equation}
\int_{M}f\left\|\nabla_H\Phi\right\|dV_\mu=\int_{\mathbb{R}}\int_{\Phi^{-1}(t)}f_{\mid\Phi^{-1}(t)}dA^{t}_{\mu}dt,
\end{equation}
where $dA^{t}_{\mu}$ is the induced sub-Riemannian area form on the surface $\Phi^{-1}(t)$. 
This formula, which involves only sub-Riemannian quantities, can be obtained by its classical Riemannian counterpart (cf.\ Proposition~\ref{prop:coareaSR} and its proof).

Finally, the \emph{sub-Riemannian mean curvature of the surface $S$ associated with $\mu$} is the smooth function $\mathcal{H}_\mu$ defined as
\begin{equation}\label{eq:Hmu}
\mathcal{H}_\mu=-\mathrm{div}_{\mu}\left(\nabla_{H}\delta\right).
\end{equation} 
where $\delta$ denotes the sub-Riemannian signed distance function from the surface $S$, cf.\ Definition~\ref{def:dsigned} for more details.

We stress that \eqref{eq:Hmu} defines $\mathcal{H}_\mu$ not only on $S$ but in a neighborhood of the surface. 
In particular, the derivative $N\left(\mathcal{H}_\mu\right)$ has a meaning. 

When $\mu = \Popp$ the measure associated with the Popp's volume form $\omega\wedge \mathrm{d}\omega$, we omit the dependence on the measure in the notation. 
For $\mathcal{H}:=\hor_{\Popp}$ we obtain the following expression in terms of a local frame (cf. Lemma~\ref{lem: H})
\begin{equation} 
\mathcal{H}=-X_{1}X_{1}\delta-X_{2}X_{2}\delta-c_{12}^{2}\left(X_{1}\delta\right)+c_{12}^{1}\left(X_{2}\delta\right).
\end{equation}

The main result of this paper is the following  formula for the measure of the localized half-tubular neighborhood. 
Namely, for $U\subset S$ we consider the subset $U_\eps$ of all points $x\in M\setminus \Omega$ for which the curve realizing the distance of the point $x$ from the surface $S$ has an end point in $U$.

\begin{theorem}
\label{thm:sviluppo}
Let $\left(M,\distr,g\right)$ be a contact three-dimensional sub-Rieman\-nian manifold equipped with a smooth measure $\mu$. 
Let $S\subset M$ be an embedded smooth surface bounding a closed region $\Omega$ and let $U\subset S$ be an open and relatively compact set such that its closure $\overline{U}$ does not contain characteristic points. The volume of the localized half-tubular neighborhood $U_\eps$, is smooth with respect to $\eps$ and satisfies for $\eps\to 0$: 
\begin{equation}
\label{eq:expansion}
 \mu\left(U_\eps\right)=\sum^3_{k=1}a_k(U,\mu)\frac{\eps^k}{k!}+o(\eps^3)
\end{equation}
where the coefficients $a_k=a_k(U,\mu)$ have the following expressions:
\begin{eqnarray}
\label{eq:coefficienti}
a_1=\int_UdA_{\mu}, & a_2=-\displaystyle{ \int_U\hor_\mu dA_{\mu}},& a_3=\int_U\left(-N\left(\mathcal{H}_\mu\right)+\mathcal{H}_\mu^2\right)dA_{\mu}.
\end{eqnarray}
Here $dA_{\mu}$ is the sub-Riemannian area measure on $S$ induced by $\mu$, $\mathcal{H}_\mu$ is the mean sub-Riemannian curvature of $S$ with respect to $\mu$ and $N$ is the sub-Riemannian normal to the surface.
\end{theorem}

\begin{remark}
Since our result is local, the requirement on $S$ to be the boundary of a regular closed region $\Omega$ can be relaxed. 
In fact, one can consider any smooth surface $S$ and it is always possible to build a local half-tubular neighborhood around a subset $U\subset S$ open, oriented and relatively compact such that $\overline{U}$ does not contain any characteristic point. This can be done in a similar way as in \cite{ritore21}.
\end{remark}

\begin{remark}
In contrast to the Euclidean case presented in \eqref{eq:steiner}, the formula \eqref{eq:expansion} in Theorem~\ref{thm:sviluppo} is not a polynomial in $\eps$. 
In the Euclidean case the polynomial expansion is related to a specific choice of the volume, the Lebesgue one. 
In the sub-Riemannian case, even for the choice of the natural volume this is not true.
Indeed, as proved in \cite{balogh15}, in the first Heisenberg group $\mathbb{H}$ equipped with the Popp's measure $\Popp$ the volume of $U_\varepsilon$ is analytic in $\eps$. 
Considering for instance as $U$ an annulus around the origin in $S$ the horizontal plane, one can compute all terms of the Taylor series and check that this is not a polynomial, see \cite[Example~5.1]{balogh15}.
\end{remark}

\begin{remark}
The sub-Riemannian mean curvature $\mathcal{H}$ can be equivalently defined as the limit of the Riemannian mean curvatures $H^{\eps}$ of $S$ with respect to the Riemannian extension $g^{\eps}$ converging to the sub-Riemannian metric imposing that the Reeb vector field is orthogonal to $\distr$ and has norm $1/\eps$. 
The same conclusion holds also for $\mathcal{H}_\mu$ when the smooth measure $\mu$ is associated to a Riemannian extension which is not necessarily defined through the Reeb vector field.
This result is proved in \cite{BTV17,BTV20} for surfaces in the Heisenberg group equipped with the Popp's measure. See Section~\ref{s:mean} for the analogue statement in the general case.
\end{remark}

To compute the coefficients of expansion~\eqref{eq:expansion} one a priori  needs the knowledge of the explicit expression of the sub-Riemannian distance, which in general is not possible. We provide a formula which permits to compute those coefficients only in terms of a function $f$ locally defining $S$. We state the result in the case $\mu=\Popp$ hence we omit the dependence on the measure in the notation.
\begin{proposition} 
\label{prop:a3f}
Under the assumptions of Theorem~\ref{thm:sviluppo}, let us suppose that it is assigned the Popp's measure $\Popp$ and that $S$ is locally defined as the level set of a smooth function $f:M\to\R$ such that $\nabla_Hf|_U \neq 0$ and $\langle\nabla_Hf,\nabla_H\delta\rangle|_U>0$. 
Then, we have the following formulas for $a_2$ and $a_3$ in expansion~\eqref{eq:expansion}:
 \begin{equation}
 a_2=-\int_U\mathrm{div}_{}\left(\frac{\nabla_Hf}{\|\nabla_Hf\|} \right)dA_{};
 \end{equation}
\begin{equation}
a_3=\int_U \left[2 X_{S}\left(\dfrac{X_0f}{\left\|\nabla_Hf\right\|} \right)-\left(\dfrac{X_0f}{\left\|\nabla_Hf\right\|}\right)^{2}-\kappa -\left\langle\Tor\left(X_0,X_S\right),\dfrac{\nabla_Hf}{\left\|\nabla_Hf\right\|}\right\rangle\right]dA_{}
\end{equation}
with
$\kappa=\left\langle R(X_1,X_2)X_2,X_1\right\rangle$
where $R$ and $\Tor$ are the curvature and the torsion operators associated with the Tanno connection and the operative expression for the characteristic vector field is 
\begin{equation}
X_S=\frac{X_2f}{\left\|\nabla_Hf\right\|}X_1-\frac{X_1f}{\left\|\nabla_Hf\right\|}X_2.
\end{equation} 
\end{proposition}
We stress that it is not a priori clear that it is possible to write such an expression since, looking for example at formula \eqref{eq:coefficienti}, the coefficient $a_3$ depends on a derivative of $\mathcal{H}$ along $N$, which is a vector field transversal to the surface. Hence it might depend on the value of $\mathcal{H}$ outside $S$, as it happens for higher order coefficients (cf.\ Remark~\ref{r:commenta4} for a comment in the Heisenberg group).

\begin{remark} We notice that the second order coefficient $a_2$ appearing in \eqref{eq:steiner2} and the corresponding one $a_2$ of expansion~\eqref{eq:expansion} are both integral of a mean curvature\footnote{The extra factor $2$ in the sub-Riemannian formula is due to the fact that in the Euclidean case one defines the mean curvature as one half of the sum of the two principal curvatures.}.
One may then wonder if the same analogy holds for the third order coefficient $a_3$. 
Namely, if the coefficient $a_3$ of expansion~\eqref{eq:expansion} is the integral of a suitably defined sub-Riemannian Gaussian curvature of the surface
appearing in \eqref{eq:steiner2}. 

In \cite{BTV17,BTV20} the authors define the sub-Riemannian Gaussian curvature $\mathcal{K}_S$ of a surface $S$ in the Heisenberg group $\mathbb{H}$ through Riemannian approximations.  
The expression of the limit, written in our notation is 
\begin{equation}
\label{eq:gaussiancurv}
\mathcal{K}_S=X_S(X_0\delta)-(X_0\delta)^2,
\end{equation}
It can be checked that the integral of this quantity does not correspond to the third coefficient $a_3$, that  thanks to Proposition~\ref{prop:a3f} rewrites as follows
\begin{equation}
\label{eq:a3ritore}
a_3=\int_U2X_S\left(X_0\delta\right)-\left(X_0\delta\right)^2\,dA.
\end{equation}
being $\kappa=0$ and $\Tor$ the null operator in the Heisenberg group.
Our result agrees with the expression obtained by Ritor\' e in \cite{ritore21}.
\end{remark}

We exploit Proposition~\ref{prop:a3f} to provide some examples. 
One can specialize the result in the case of surfaces embedded in some particular three-dimensional contact manifolds, considering the metric invariants $\chi$ and $\kappa$ defining the local geometry of the sub-Riemannian structure. 
This coefficients correspond to the torsion and the horizontal scalar curvature of a natural affine connection and appear in the coefficients.
For a three-dimensional contact manifold with $\chi=0$ and equipped with the Popp's measure $\Popp$, we have that
\begin{equation}
a_3=\int_U 2X_{S}(X_0\delta)-\left(X_0\delta\right)^{2}-\kappa\, dA_{}.
\end{equation} 
Moreover, chosen in terms of a canonical basis $X_1,X_2$ for a left invariant sub-Riemannian contact structure on a Lie group with $\chi\neq0$ (see for instance see \cite[Prop. 17.14]{thebook}) and again equipped with the Popp's measure $\Popp$, it holds that
\begin{equation}\label{eq:a33}
a_3=\int_U 2X_{S}(X_0\delta)-(X_0\delta)^{2}-\kappa+\chi\left(\left(X_{1}\delta\right)^{2}-\left(X_{2}\delta\right)^{2}\right)\,  dA_{}.
\end{equation}
We stress that in the Heisenberg group $\kappa=\chi=0$, recovering formula \eqref{eq:a3ritore}, while in general on contact manifolds the curvature and torsion operators associated with the Tanno connection are not necessarily vanishing.

Finally, we compute expansion~\eqref{eq:expansion} in the case of a surface embedded in $\mathbb{H}$ with rotational symmetry with respect to the $z$-axis. 
In addition, we show that for ``model'' surfaces in the model spaces $SU(2)$ and $SL(2)$ (cf.\ \cite{BBCH21}, see also \cite{BH22}), it holds $a_2=a_3=0$ as for the horizontal plane in $\mathbb{H}$.

%


\subsection{Structure of the paper}
In Section~\ref{s:preliminaries} we recall some preliminary notions about three-dimensional sub-Rieman\-nian contact manifolds, the Tanno connection and the sub-Riemannian distance. 
In Section~\ref{s:coarea} we define the local half-tubular neighborhood $U_\eps$ and discuss about the sub-Riemannian mean curvature.
Furthermore, we prove a sub-Riemannian version of the coarea formula
 in order to derive a formula for the measure of $U_\eps$.
Section~\ref{s:localized} is devoted to the study of the Jacobian of the exponential map at a fixed time with the goal to derive a handy expression for the computations. 
The proof of Theorem~\ref{thm:sviluppo} is contained in Section~\ref{s:expansion}. 
In Section~\ref{s:operative}, we present an operative way to compute the coefficients in expansion~\eqref{eq:expansion} without knowing explicitly the distance function. Moreover, we specialize our result in the case of three-dimensional contact manifolds with particular values of the geometric invariants $\chi$ and $\kappa$.
In section~\ref{Applications} we show some applications of the results to rotational surfaces in the Heisenberg group and model surfaces in the spaces $SU(2)$ and $SL(2)$.\\

{\bf Acknowledgments.} 
Davide Barilari acknowledges the support granted by the European Union – NextGenerationEU  Project ``NewSRG - New directions in Sub-Riemannian Geometry'' within the Program STARS@UNIPD 2021.
The authors are 
are partially supported by the INdAM--GNAMPA 2022 Project ``Analisi Geometrica in Strutture Sub-Riemanniane'' CUP\_E55F22000270001.


\section{Preliminaries}
\label{s:preliminaries}

In this section we recall some definitions and properties regarding  three-di\-me\-nsion\-al contact sub-Riemannian manifolds, the canonical Tanno connection and the sub-Riemannian distance. For a more detailed and general presentation see \cite{thebook}.

\subsection{Three-dimensional contact sub-Riemannian manifolds}
\begin{definition}
A \emph{three-dimensional contact sub-Riemannian manifold} $(M,\distr,g)$ is a smooth manifold $M$ of dimension three equipped with a two dimensional vector distribution $\distr$ that is the kernel of a one-dimensional form $\omega$ such that $\omega\wedge \mathrm{d}\omega \neq 0$. We endow the distribution $\distr$ with a smooth metric $g$.
\end{definition}
Up to multiplication by a never vanishing function, one can always normalize the contact form with respect to $g$ by requiring that $\mathrm{d}\omega|_{\distr}$ coincides with the volume form defined by the metric $g$ on $\distr$, that is $\mathrm{d}\omega(X_{1},X_{2})=1$ for every positively oriented orthonormal frame of $\distr$. In the sequel we always assume this normalization.

The conditions $\ker \omega=\distr$ and $\omega\wedge \mathrm{d}\omega \neq 0$ are equivalent to saying that for any pair of locally defined and linearly independent vector fields $X$ and $Y$ in $\distr$, the Lie bracket $[X,Y]$ is not contained in $\distr$. In other words $\distr$ is Lie bracket generating.
 
 Any smooth vector field tangent to $\distr$ (resp.\ any absolutely continuous curve whose velocity lies a.e.\ in $\distr$) is called {\it horizontal vector field} (resp.\  {\it horizontal curve}). 
 
 The Lie bracket generating condition implies that there exists at least one horizontal curve connecting any two points in $M$ (Chow-Rashevkij theorem). 

The \emph{Reeb vector field} $X_{0}$ is the unique vector field such that
\begin{equation}\label{eq:reeb}
i_{X_{0}}\omega=1,\qquad i_{X_{0}}\mathrm{d}\omega=0,
\end{equation}
where $i_{X_{0}}$ denotes the interior product. 
For every positively oriented local orthonormal frame $X_1, X_2$ of $\distr$ it holds that ${X_1, X_2,X_0}$ is a local frame for $TM$ satisfying the identity $[X_2,X_1]=X_{0}\mod\distr$.

Finally we can introduce on $M$ a Riemannian metric such that ${X_1, X_2,X_0}$ is a local orthohormal frame for $TM$, i.e., extending the metric $g$ defined on $\distr$ by declaring that $X_{0}$ has unit norm and is orthogonal to $\distr$. 
The Riemannian metric will be denoted with $\left\langle\cdot,\cdot\right\rangle$ and the corresponding induced norm with $\| \cdot \|$.
\begin{definition}
Given a smooth function $f:M\to \R$, we denote with $\nabla_{H}f$ its horizontal gradient, which is the unique horizontal vector field satisfying the identity $\mathrm{d}f(v)=g(\nabla_{H}f,v)$ for every $v\in \distr$.
\end{definition}
Denoting by $\nabla f$ the Riemannian gradient of $f$ with respect to the ambient Riemannian metric, in terms of a local orthonormal frame ${X_1, X_2,X_0}$ chosen as above, the expressions
for the gradients are the following:
\begin{gather}
\label{eq:gradient}
\nabla f=(X_{1}f)X_{1}+(X_{2}f)X_{2}+(X_{0}f)X_{0}, \\  \nabla_{H}f=(X_{1}f)X_{1}+(X_{2}f)X_{2}.
\end{gather}

From the properties defining the Reeb vector field \eqref{eq:reeb} and the normalization of the contact form we get that for every positively oriented orthonormal frame of $\distr$ we have the Lie brackets can be written as follows
\begin{align}
\left[X_{2},X_{1}\right] & =c_{12}^{1}X_{1}+c_{12}^{2}X_{2}+X_{0},\nonumber \\
\left[X_{1},X_{0}\right] & =c_{01}^{1}X_{1}+c_{01}^{2}X_{2},\label{eq:structurecoeff}\\
\left[X_{2},X_{0}\right] & =c_{02}^{1}X_{1}+c_{02}^{2}X_{2}.\nonumber 
\end{align}
for suitable smooth functions $c_{ij}^{k}$ defined on $M$. We notice that these functions are constant if the sub-Riemannian structure is left-invariant on a Lie group.
\begin{definition}
\label{def:J}
Let $J:TM\to TM$ be the linear morphism of vector bundles defined by the identities
\begin{eqnarray}
J(X_1)=X_2, & J(X_2)=-X_1, & J(X_0)=0.
\end{eqnarray}
 Notice that $g(X,JY)=-\mathrm{d}\omega(X,Y)$, for every $X,Y$ in $TM$, hence $g(X,JX)=0$, and that $J|_{\distr}:\distr\to \distr$ preserves the sub-Riemannian metric.
\end{definition}

The following lemma follows from a direct computation.
\begin{lemma}
\label{lem:rotation}Let $Y_1,Y_2=JY_1$ and $X_1, X_2=JX_1$ be two pairs of horizontal
vector fields satisfying
\begin{equation}
\left(\begin{array}{c}
Y_{1}\\
Y_{2}
\end{array}\right)=\left(\begin{array}{cc}
\text{\ensuremath{\cos\theta}} & -\sin\theta\\
\sin\theta & \cos\theta
\end{array}\right)\left(\begin{array}{c}
X_{1}\\
X_{2}
\end{array}\right)
\end{equation}
for some smooth function $\theta:M\to \R$. Then we have
$$\left[Y_{2},Y_{1}\right]=\left[X_{2},X_{1}\right]-X_{1}\left(\theta\right)X_{1}-X_{2}\left(\theta\right)X_{2}.$$
\end{lemma}

\subsection{Tanno connection} 
Let us now introduce the Tanno connection, which is a canonical connection on contact manifold, see for instance  \cite{tanno, abrcontact}. 
\begin{definition}
\label{def:Tanno}
The Tanno connection $\nabla$ is the unique linear  connection on $TM$ satisfying
\begin{itemize}
\item[(i)] $\nabla g=0$, $\nabla X_0=0$;
\item[(ii)] $\Tor(X,Y)=-\left\langle X,JY\right\rangle X_0=\mathrm{d}\omega(X,Y)X_0$ for all $X,Y\in\distr$;
\item[(iii)] $\Tor(X_0,JX)=-J\Tor(X_0,X)$ for any vector field $X$ on $M$. 
\end{itemize}
where $\Tor$ denotes the Torsion of the connection $\nabla$.
\end{definition}
\begin{remark}
The Tanno connection $\nabla$ commutes with the operator $J$, i.e., $$\nabla_XJY=J\nabla_XY.$$ 
Indeed, $J\nabla_XY$ is horizontal by definition of $J$, while $\nabla_XJY$ is horizontal since
\begin{equation}
0=X\langle JY,X_0\rangle=\langle\nabla_XJY,X_0\rangle+\langle JY,\nabla_XX_0\rangle=\langle\nabla_XJY,X_0\rangle.
\end{equation}
Moreover, from the fact that $J$ preserves the metric, we deduce that
\begin{equation}
\langle \nabla_XJY,JY\rangle=\frac{1}{2}X\langle JY,JY\rangle=\frac{1}{2}X\langle Y,Y\rangle=\langle \nabla_XY,Y \rangle=\langle J\nabla_XY,JY\rangle.
\end{equation}
Finally, since $\langle Y,JY\rangle=0$ one obtains that
\begin{equation}
0=X\langle Y,JY\rangle=\langle\nabla_XY,JY\rangle+\langle Y,\nabla_XJY\rangle
=\langle -J\nabla_XY,Y\rangle+\langle \nabla_XJY,Y\rangle.
\end{equation}
\end{remark}
\begin{lemma}
For every vector field $X$ on $M$, it holds
\begin{equation}
\label{eq:torsioneX0}
\left\langle \Tor\left(X_0,X\right),JX\right\rangle = \frac{1}{2}\left\langle \left[JX,X_{0}\right],X\right\rangle +\frac{1}{2}\left\langle \left[X,X_{0}\right],JX\right\rangle.
\end{equation}
\end{lemma}
\begin{proof}
Using (iii) of Definition~\ref{def:Tanno}, we have the equality 
\begin{equation}
\left\langle \Tor\left(X_0,X\right),JX\right\rangle = \left\langle \Tor\left(X_0,JX\right),X\right\rangle.
\end{equation}
On the other hand, by the definition of the torsion, we have that
\begin{align}
 \left\langle \Tor \left(X_0,X\right) ,JX \right\rangle &= \left\langle \nabla_{X_0}X + \left[ X_0,X \right] ,JX   \right\rangle ,\\
 \left\langle \Tor\left(X_0,JX\right),X\right\rangle &= \left\langle \nabla_{X_0}JX + \left[X_0,JX\right],X\right\rangle.
\end{align}
Adding the two equations one obtains the statement recalling in addition that $\left\langle \nabla_{X_0}JX,X\right\rangle = - \left\langle \nabla_{X_0}X,JX\right\rangle $, being $X$ and $JX$ orthogonal.
\end{proof}

Let $\tau:\distr \to \distr$ be the linear operator $\tau(X)=\Tor(X_0,X)$ for $X$ a vector field in $\distr$. The symmetric matrix representing the operator $\tau$ with respect to the orthonormal and positively oriented frame $X_1,X_2$ for $\distr$ satisfying \eqref{eq:structurecoeff} is
\begin{equation}
\label{eq:matricetau}
\left(\begin{array}{cc}
c_{01}^1 & \dfrac{c_{02}^1+c_{01}^2}{2}\\
\dfrac{c_{02}^1+c_{01}^2}{2}& c_{02}^2
\end{array}\right).
\end{equation}

Notice that, for a smooth function $f$, the notation $\nabla f$ is compatible with its Riemannian gradient, if one interprets $\nabla$ as the Tanno connection.

\subsection{Sub-Riemannian distance}

Recall that the sub-Riemannian distance $d_{SR}$ between two points $x,y\in M$ is defined as
\begin{equation}
\label{eq:distance}
d_{SR}(x,y)=\inf \left\{ \ell(\gamma) \mid \gamma\in \mathcal{C}_{x,y}\right\}
\end{equation}
where $\ell(\gamma)=\int_{0 }^{T} \| \dot{\gamma} (t)\|dt$ and the infimum is taken over the set $\mathcal{C}_{x,y}$ of absolutely continuous (AC) horizontal curves connecting two fixed points $x,y$, i.e.,
\begin{equation}
\label{eq:Cxy}
\mathcal{C}_{x,y}=\{\gamma\in\mathrm{AC}([0,T], M)\mid\gamma(0)=x,\gamma(T)=y,\dot{\gamma}(t)\in \distr_{\gamma(t)} \mbox{ a.e. in } [0,T]\},
\end{equation}
Notice that the length of an horizontal curve is invariant by reparametrization, for more details we refer to \cite{thebook}.

The Hamiltonian $H:T^*M\to\R$ related to the sub-Riemannian  length minimization problem in a three-dimensional contact manifold is 
\begin{equation}
H(\lambda)=H(p,x)=\frac{1}{2}\left[\langle p,X_1(x)\rangle^2+\langle p,X_2(x)\rangle^2\right],
\end{equation} 
where $p \in T^*_xM$ and $\langle\cdot ,\cdot \rangle$ denotes the duality pairing. Considering the canonical symplectic form $\sigma\in \Lambda^2(T^*M)$, the induced Hamiltonian vector field $\vec{H}$ on $T^*M$ is defined by $dH(\cdot)=\sigma(\cdot,\vec{H})$. Then, the Hamilton equation is 
\begin{equation}
\label{eq:Hamilton}
\dot{\lambda}=\vec{H}(\lambda).
\end{equation}

Given a initial covector $\lambda_0\in T^*M$, the unique solution $\lambda(t)=e^{t\vec{H}}\left(\lambda_0\right)$ to \eqref{eq:Hamilton} is called normal extremal. Moreover, $\gamma(t)=\pi(\lambda(t))$, where  $\pi:T^*M\to M$ is the canonical projection on $M$, is a locally length minimizing curve parametrized with constant speed $\sqrt{2H(\lambda_0)}$.

In a three-dimensional contact sub-Riemannian manifold all length-minimizers arise in this way (cf. \cite[Prop. 4.8]{thebook}). 
To obtain arclength parametrized length minimizers one has to consider solutions of \eqref{eq:Hamilton} with $\lambda_0\in H^{-1}(1/2)$.



\section{Half-tubular neighborhoods and  coarea formula}
\label{s:coarea}

Let $\left(M,\distr,g\right)$ be a three-dimensional contact manifold equipped with a smooth volume measure $\mu$ and let $S$ be a smooth surface embedded in $M$ bounding a closed region $\Omega$.
Let us consider a positively oriented local orthonormal frame on $\distr$ given by the vector fields $X_{1},X_{2}$, and let us extend the sub-Riemannian metric $g$ to the Riemannian one for which the Reeb vector field $X_0$ is orthogonal to the distribution and with unit norm. Moreover, let us denote with $\omega^{1},\omega^{2},\omega^{0}$ on $T^{*}M$ the dual $1$-forms to $X_1,X_2,X_0$. Notice that $\omega^{0}=\omega$ is the normalized contact form. 
Recall that  $\omega \wedge \omega^{1} \wedge \omega^{2}=\omega\wedge \mathrm{d}\omega$ is the Popp's volume form of the three-dimensional contact sub-Riemannian manifold (see for instance \cite{thebook}).

\subsection{Regularity of the distance and tubular neighborhoods}

We recall that a point $x\in S$ is said to be characteristic if $T_{x}S=\distr_{x}$ and that the set $\Gamma(S)$ of characteristic points in $S$ is closed and has zero measure on the surface $S$, see \cite{balog03}.
As a consequence, we have a well-defined characteristic foliation on $S\setminus \Gamma(S)$ defined by the intersection $\distr _x \cap T_xS$.

\begin{definition}
\label{def:dsigned}
Let us consider on $M$ the sub-Riemannian signed distance function from the surface $S$: 
\begin{equation}
\delta_{S}(p)= \begin{cases}
\text{inf}\left\{ d_{SR}\left(p,x\right)\mid x\in S \right\},&p\notin \Omega,\\
-\text{inf}\left\{ d_{SR}\left(p,x\right)\mid x\in S \right\},&p\in \Omega.
 \end{cases}
\end{equation}
\end{definition}

Our aim is to obtain results for the volume of the external sub-Riemannian $\eps$-half-tubular neighborhood of the surface $S$, that is
\begin{eqnarray}
S_{\eps}=\left\{ x\in M\setminus \Omega\mid 0<\delta_{S}(x)< \eps\right\}.
\end{eqnarray}

More precisely, we are going to study local results considering a $\eps$-half-tubular neighborhood of open and relatively compact subsets of $S$ that do not contain characteristic points.
In order to do that, we report a key result about the smoothness of the sub-Riemannian distance function from surfaces.

\begin{theorem}
\label{thm:smoothness}
Let $S$ be a smooth surface bounding a closed domain $\Omega$ in a three-dimensional contact sub-Riemannian manifold $M$. Let us consider $U\subset S$ open, relatively compact and such that  $\overline U$ does not contain characteristic points. Then
\begin{enumerate}
\item there exist $\eps >0$ and a smooth map $G:(-\eps,\eps)\times U\to M$ that is a diffeomorphism on the image and such that for all $(t,x)\in(-\eps,\eps)\times U$
\begin{eqnarray}
\delta_S\left(G(t,x)\right)=t & \mbox{and} & dG(\partial_{t})=\nabla_{H}\delta_S;
\end{eqnarray}
\item  $\delta_S$ is smooth on $G((-\eps,\eps)\times U)$,  with $\|\nabla_H\delta_S\|=1$.
\end{enumerate}
\end{theorem}
Theorem~\ref{thm:smoothness} can be found in \cite[Prop.\ 3.1]{FPR20} and in \cite[Thm.\ 3.3]{Rossi22} for compact hypersurfaces and for submanifolds of arbitrary codimension, respectively, without characteristic points, embedded in sub-Riemannian manifolds. 
The version presented here refers to \cite[Thm.\ 3.7]{Rossi22}, that, in turn, is a refinement of the previously mentioned results for the signed distance from non characteristic hypersurfaces.
Our goal is to study the measure of the following subsets of $S_\eps$.
\begin{definition}
 Let $U\subset S$ be open, relatively compact and such that the closure does not contain characteristic points. For $\eps>0$ we define the (external) local half-tubular neighborhood of $S$ relatively to $U$ as
\begin{equation}
\label{eq:Ur}
U_{\eps}=G\left((0,\eps)\times U\right) \subset S_\eps,
\end{equation}
where $G:(-\eps,\eps)\times U\to M$ is the diffeomorphism of Theorem~\ref{thm:smoothness}.
\end{definition}
\begin{remark}
 Notice that the set $U_{\eps}$ is well-defined thanks to Proposition~\ref{p:GI}. 
In fact, the map $G$ is unique and explicitly characterized in \eqref{eq:G}. 
The localized half-tubular neighborhood $U_\eps$ can be described as the union of the length-minimizing curves contained in $S_\eps$ realizing the distance from the surface $S$ with an endpoint on $U$. The description through the diffeomorphism $G$ guarantees that $U_\eps$ is open and smooth in $M$, hence measurable. 
\end{remark}
\begin{proof}[Proof of Theorem~\ref{thm:smoothness}]
We report only a sketch of the proof because it follows verbatim from that of \cite[Thm.\ 3.7]{Rossi22} with only minor adjustments to adapt  it to our setting.
The set $\Gamma(S)$ of characteristic points in $S$, is closed, thus we can consider $\widetilde{U}$ an open neighborhood of $\overline{U}$, the compact closure of $U$ in $S$, that does not contain characteristic points.
Let us define the annihilating bundle
$$\mathcal{A}\widetilde{U}=\{(x,\lambda)\in T^*M \mid x\in \widetilde{U},\langle\lambda,T_xS\rangle=0\}.$$
The exponential map $E:\mathcal{A}\widetilde{U}\to M$ defined by
$E(\lambda)=\pi\circ  e^{\vec{H}(\lambda)}$
is a local diffeomorphism  at every $(x,0)\in\overline{U}$.
Since $\overline{U}$ is compact, there exists $\eps>0$ such that the restriction of $E$ to the open set 
$$\{(x,\lambda)\mid x\in U, 0< \sqrt{2H(\lambda)}<\eps\}\subset\mathcal{A}\widetilde{U},$$
is a diffeomorphism onto its image.
Finally, being $U$ oriented, there exists a smooth outward pointing non-vanishing section of $\mathcal{A}\widetilde{U}$, i.e., $\lambda:\widetilde{U} \to \mathcal{A}\widetilde{U}$ such that $E(\lambda(x))\notin \Omega$ for every $x\in U$. Without loss of generality, we can suppose that $2H(\lambda(x))=1$, then the map we are looking for is defined as 
\begin{equation}
G(t,x)=E(t\lambda(x)). \qedhere
\end{equation} 
\end{proof}

\begin{remark}
The requirement on $U$ to have compact closure and without characteristic points is essential in order to define the diffeomorphism in Theorem~\ref{thm:smoothness}.
The absence of characteristic points guarantees the smoothness of the map, while the hypothesis of compactness of the closure of $U$ guarantees the existence of a minimum $\eps>0$ for which we can define the map $G$. 

%
\end{remark}

\subsection{Transversality conditions}

Let us fix $U\subset S$ open, relatively compact and such that its closure does not contain any characteristic point. 
Thanks to Theorem~\ref{thm:smoothness}, for a sufficiently small $\eps>0$ the set $G\left((-\eps,\eps)\times U\right)$ is a open submanifold in $M$. 
For the sake of simplicity, from now on we refer to $G\left((-\eps,\eps)\times U\right)\subset M$ as ambient space, and we denote with $\delta$ the sub-Riemannian signed distance function from $U$. 
In this way $\delta$ is a smooth function and since $\nabla_H\delta\neq 0$, we can treat $\delta$ as a defining function for $U$, i.e., $U=\delta^{-1}(0)$ and $U_\eps=\delta^{-1}(]0,\eps[)$.

In coordinates, identity $\left\|\nabla_H\delta\right\|^2=1$  of Theorem~\ref{thm:smoothness} is written as
\begin{equation}
\label{eq:gradH=1}
(X_1\delta)^2+(X_2\delta)^2=1,
\end{equation}
where we recall that $\nabla_H\delta=(X_1\delta)X_1+(X_2\delta)X_2$.

One can also introduce on $U$ a characteristic vector field as follows
\begin{equation}
\label{eq:XS}
X_S=(X_2\delta)X_1-(X_1\delta)X_2,
\end{equation}
and consider the orthogonal basis on $TU$ composed by $X_S$ and
\begin{equation}
\label{eq: YS}
Y_{S}=\left(X_{0}\delta\right)\nabla_{H}\delta-X_{0}.
\end{equation}

\begin{remark}
Notice that $X_{S}$ is a smooth vector field of unit norm, and it spans $\distr _x \cap T_xS$ for every $x\in U$. 
Moreover, we recall that if the surface contains characteristic points then $X_S$ is still well-defined through \eqref{eq:XS} replacing the sub-Riemanian distance function $\delta$ with a smooth function $f:M\to \R$ that locally defines the surface as its zero level set, i.e., $X_S=(X_2f)X_1-(X_1f)X_2$.
In this case, the resulting vector field is not of unit norm and it vanishes at characteristic points.
\end{remark}
The distance from the surface $\delta$ restricted to $U_\eps$ can be rewritten as follows
\begin{equation}
\label{eq:thed}
\delta(x)=\inf\left\{ \ell(\gamma)=\int_{0 }^{T} \| \dot{\gamma} (t)\|dt  \mid \gamma \in \mathcal{C}_{x_0,x},\, x_0\in U \right\}, \qquad \mbox{for } x\in U_\eps,
\end{equation}
where $\mathcal{C}_{x_0,x}$ is defined in \eqref{eq:Cxy}.
Let us consider $\gamma:[0,\delta(x)]\to M$ an arclength parametrized curve that realizes the distance \eqref{eq:thed} from $x=\gamma(\delta(x))$ to the surface. A lift $\lambda:[0,\delta(x)]\to T^*M$ of $\gamma$ that is a solution of \eqref{eq:Hamilton} has to satisfy the transversality conditions given by  the Pontryagin Maximum Principle (PMP) for optimal control problems with constraints on initial and terminal points (see \cite[Thm.~12.13]{AS04}). 
More precisely, the condition on $\lambda(0)=\lambda_0$ is 
\begin{equation}
\label{eq:transversality}
\left\langle \lambda_0,v\right\rangle =0, \qquad \forall v\in T_{\gamma(0)}U. 
\end{equation}
Exploiting this fact, one can compute the initial covector that identifies the arclength parametrized curve that realizes the distance of a fixed point in $U_\eps$ from the surface.

\begin{proposition}
\label{prop: covector}
Let $\gamma:[0,T]\to M$ be an arclength parametrized curve with $\gamma(0)\in U$ that minimizes the distance from $\gamma(T)\in U_\eps$ to $U$, and let $\lambda:[0,T]\to T^*M$ be the corresponding lift solving \eqref{eq:Hamilton}. Then the initial covector is 
\begin{equation}
\label{eq:CCovector}
\lambda_0=\left(X_{1}\delta\right)\omega^{1}+\left(X_{2}\delta\right)\omega^{2}+\left(X_{0}\delta\right)\omega \in T^*_{\gamma(0)}M.
\end{equation}
Moreover, $\gamma(t)=\pi \circ e^{t\vec{H}}(\lambda_0)$ and $\dot{\gamma}(0)=\nabla_H\delta$.
\end{proposition}

\begin{proof}
Let us consider the local ortogonal basis on $U$ composed by $X_S,Y_S$ defined in \eqref{eq:XS} and \eqref{eq: YS}.
The transversality condition \eqref{eq:transversality} applied to the vector field $X_{S}$ is written as follows
\begin{equation}
0=\left\langle \lambda_0,X_{S}\right\rangle=\left(X_{2}\delta\right)\left\langle \lambda_0,X_{1}\right\rangle -\left(X_{1}\delta\right)\left\langle \lambda_0,X_{2}\right\rangle 
\end{equation}
from which one gets $\left(\left\langle \lambda_0,X_{1}\right\rangle ,\left\langle \lambda_0,X_{2}\right\rangle \right)=c\left(X_{1}\delta,X_{2}\delta\right)$
with $c\in\mathbb{R}$. 
Turning the attention to  $Y_{S}$,
\begin{equation}
0=\left\langle \lambda_0,Y_{S}\right\rangle =\left(X_{0}\delta\right)\left(X_{1}\delta\right)c\left(X_{1}\delta\right)+\left(X_{0}\delta\right)\left(X_{2}\delta\right)c\left(X_{2}\delta\right)-\left\langle \lambda_0,X_{0}\right\rangle.
\end{equation}
Hence, from \eqref{eq:gradH=1} we have that  $ \left\langle \lambda_0,X_{0}  \right\rangle =c(X_{0}\delta).$ That means that the covector $\lambda_0$ is a multiple of  
 \begin{equation}
\left(X_{1}\delta\right)\omega^{1}+\left(X_{2}\delta\right)\omega^{2}+\left(X_{0}\delta\right)\omega \in T^*_{\gamma(0)}M.
\end{equation}
Finally, writing in coordinates the Hamiltonian system \eqref{eq:Hamilton}, we have that 
\begin{eqnarray}
\dot{x}=\dfrac{\partial H}{\partial p},&\dot{p}=-\dfrac{\partial H}{\partial x}.
\end{eqnarray} 
Then, we deduce that 
\begin{equation}
\dot{\gamma}(0)=\left\langle \lambda_0,X_{1}\right\rangle X_1+\left\langle \lambda_0,X_{2}\right\rangle X_2=c\left(X_{1}\delta\right)X_1+c\left(X_{2}\delta\right)X_2.
\end{equation}
Requiring $\gamma$ to have unit speed, \eqref{eq:gradH=1} implies $c=1$ and we conclude that 
\begin{equation}
\dot{\gamma}(0)=\nabla_{H}\delta(x),  \quad  \lambda_0=\left(X_{1}\delta\right)\omega^{1}+\left(X_{2}\delta\right)\omega^{2}+\left(X_{0}\delta\right)\omega. \qedhere
\end{equation}
\end{proof}

\begin{remark}
Given $x\in U$, there is a unique arclength parametrized curve $\gamma_x:[0,T]\to M$, such that $\gamma(0)=x$ and that realizes the distance of $\gamma(T)$ from $U$. Namely, $\gamma_x(t)=\pi\circ e^{t\vec{H}}(\lambda_0)$ with $\lambda_0\in T^*_xM$ prescribed by \eqref{eq:CCovector}.
Therefore, Theorem~\ref{thm:smoothness} together with Proposition~\ref{prop: covector} provides a description of $U_\eps$ both as the disjoint union of copies of $U$ and as the disjoint union of arclength parametrized minimizing curves escaping from $U$.
\end{remark}

In conclusion, let us introduce the smooth map $\lambda_0:U\to T^*M$ as follows:
\begin{equation}
\label{eq:lambda0}
\lambda_0(x)=\left(X_{1}\delta\right)\omega^{1}+\left(X_{2}\delta\right)\omega^{2}+\left(X_{0}\delta\right)\omega \in T^*_{x}M.
\end{equation}

We have proved the following
\begin{proposition}\label{p:GI}
The map $G$ of Theorem~\ref{thm:smoothness} is uniquely characterized as 
\begin{equation}
\label{eq:G}
G(t,x)=\gamma_x(t)=\pi\circ e^{t\vec{H}}(\lambda_0(x))
\end{equation} 
for every $t\in (-\eps,\eps)$ and $x\in U$, with $\lambda_{0}$ defined by \eqref{eq:lambda0}. 
\end{proposition}

\subsection{The mean sub-Riemannian curvature}\label{s:mean}
Given a smooth volume form $\mu$ on $M$, the divergence of a vector field $X$ with respect to a measure $\mu$ is the function denoted $\text{div}_{\mu}X$ such that $\mathcal{L}_{X}\mu=(\text{div}_{\mu}X)\mu$ (where $\mathcal{L}_{X}$ is the Lie derivative along the vector field $X$). 

\begin{definition}
\label{def: H}
The \emph{sub-Riemannian mean curvature with respect to a smooth measure $\mu$} of a surface $U$ embedded in a three-dimensional contact manifold and that does not contain characteristic points, is the smooth function $\mathcal{H}_\mu:U_\eps\to\R$
\begin{equation}
\mathcal{H}_\mu=-\mathrm{div}_{\mu}\left(\nabla_{H}\delta\right).
\end{equation}
Moreover, the \emph{sub-Riemannian mean curvature} is the smooth function $\mathcal{H}:U_\eps\to \R$ defined as $\mathcal{H}=\mathcal{H}_{\nu}$, where $\Popp$ is the canonical measure on $M$ associated with the Popp's volume $\omega\wedge \mathrm{d}\omega=\omega\wedge\omega^1\wedge \omega^2$.
\end{definition}

\begin{remark}
\label{rem:H}
Let us consider $f:U_\eps \to \R$ a smooth defining function for $U$, i.e., such that $U=f^{-1}(0)$ and $\mathrm{d}f\neq 0$ on $U$.
Then $\left\|\nabla_{H}f\right\|\neq 0$ on $U$ (since $U$ does not contain characteristic points), and it holds that
\begin{equation}
\label{eq:HU}
\mathcal{H}|_U=-\mathrm{div}_{\Popp}\left(\frac{\nabla_{H}f}{\left\|\nabla_{H}f\right\|}\right).
\end{equation}
Actually, identity \eqref{eq:HU} is valid if the orientation of the vector field $\nabla_{H}f$ coincides with that of $\nabla_H\delta$, in other words if $\langle\nabla_Hf,\nabla_H\delta\rangle>0$. 
Otherwise the equality holds with the opposite sign.
Notice that in general the right hand side of \eqref{eq:HU} is well defined on $S$ but outside the surface it still depends on the choice of the function $f$.
Moreover, we recall that since every smooth measure $\mu$ can be interpreted as a multiple of the Popp's measure by a smooth density, i.e., $\mu= h \Popp$ with $h:M\to \R$ non vanishing. 
From the properties of the divergence operator we have that
$$\mathrm{div}_{h \Popp}X=\mathrm{div}_{\Popp}X+\frac{Xh}{h},$$
 and we obtain the formula
 \begin{equation}
\label{eq:Hm}
\mathcal{H}_\mu =- \mathrm{div}_{h \Popp}(\nabla_H\delta)= \mathcal{H}-\frac{\nabla_H\delta (h)}{h}.
\end{equation}
\end{remark}

\begin{lemma}
\label{lem: H}
The expression in coordinates of the sub-Riemannian mean curvature $\mathcal{H}:U_\eps\to \R$ with respect to the Popp measure $\Popp$ is: 
\begin{equation} 
 \label{eq: H}
\mathcal{H}=-X_{1}X_{1}\delta-X_{2}X_{2}\delta-c_{12}^{2}\left(X_{1}\delta\right)+c_{12}^{1}\left(X_{2}\delta\right).
\end{equation}
In particular, in terms of \eqref{eq:structurecoeff},
\begin{equation}
\label{eq:divergenceX}
\mathrm{div}_{\Popp}\left(X_{1}\right)=c_{12}^{2}\quad\text{ and }\quad\mathrm{div}_{\Popp}\left(X_{2}\right)=-c_{12}^{1}.
\end{equation}
\end{lemma}

\begin{proof}
Applying the linearity of the divergence and the Leibnitz rule for
which $\mathrm{div}(aX)=X(a)+a\mathrm{div}(X)$
for $a$ a smooth function and $X$ a smooth vector field,
\begin{align}
\mathcal{H} & =-\mathrm{div}_{\Popp}\left(\left(X_{1}\delta\right)X_{1}+\left(X_{2}\delta\right)X_{2}\right)\\
 & =-X_{1}X_{1}\delta-\left(X_{1}\delta\right)\mathrm{div}_{\Popp}\left(X_{1}\right)-X_{2}X_{2}\delta-\left(X_{2}\delta\right)\mathrm{div}_{\Popp}\left(X_{2}\right)
\end{align}
Now, considering that for every  $i=0,1,2$
\begin{equation}
\mathcal{L}_{X}\omega^{i}=\sum_{j=0}^{2}\left\langle \mathcal{L}_{X}\omega^{i},X_{j}\right\rangle \omega^{j}=\sum_{j=0}^{2}\left\langle \omega^{i},\mathcal{L}_{-X}X_{j}\right\rangle \omega^{j}=\sum_{j=0}^{2}\left\langle \omega^{i},\left[-X,X_{j}\right]\right\rangle \omega^{j},
\end{equation}
where $\mathcal{L}_{X}$ is the Lie derivative along the vector field $X$, we obtain that
\begin{align}
\text{div}_{\Popp}\left(X\right)\omega\wedge \omega^1\wedge\omega^2 & =\left(\mathcal{L}_{X}\omega\right)\wedge\omega^{1}\wedge\omega^{2}+\omega\wedge\left(\mathcal{L}_{X}\omega^{1}\right)\wedge\omega^{2}+\omega\wedge\omega^{1}\wedge\left(\mathcal{L}_{X}\omega^{2}\right)\\
 & =\sum_{i=0}^{2}\left\langle \omega^{i},\left[X_{i},X\right]\right\rangle \omega\wedge\omega^{1}\wedge\omega^{2}.
\end{align}
Therefore $\text{div}_{\Popp}\left(X_{1}\right)=c_{12}^{2}$ and $\text{div}_{\Popp}\left(X_{2}\right)=-c_{12}^{1}$ which proves the statement.
\end{proof}

\begin{remark}
\label{rem:Hmu}
Following the proof of Lemma~\ref{lem: H}, one can obtain an explicit expression of $\mathcal{H}_\mu$ with $\mu=h\Popp$ a smooth measure where $h$ is a positive smooth density.
More precisely, let $X_\theta$ be a vector field transverse to the distribution $\distr$ such that $\mu(X_1,X_2,X_\theta)=1$.
Let us denote by $\{\omega^1,\omega^2,\omega^\theta\}$ the dual basis to $\{X_{1},X_{2},X_\theta\}$, then $\mu=\omega^\theta\wedge\omega^1\wedge\omega^2$.
Furthermore,  let us consider the structure functions associated to $X_1,X_2,X_\theta$ defined as in \eqref{eq:structurecoeff} (notice that $X_{\theta}$ here is transverse but not necessarily the Reeb vector field). Namely, we set
\begin{equation}
c^k_{ij}=\langle \omega^k , [X_j,X_i]\rangle,\qquad i,j,k=1,2,\theta.
\end{equation}
Replacing $X_0$ and $\omega$ with $X_\theta$ and $\omega^\theta$ in the proof of Lemma \ref{lem: H}, one obtains
\begin{equation}
\label{eq:divergenceXmu}
\mathrm{div}_{\mu}\left(X_{1}\right)=c_{12}^{2}-c^\theta_{\theta1},\qquad \quad\mathrm{div}_{\mu}\left(X_{2}\right)=-c_{12}^{1}-c^\theta_{\theta2}. 
\end{equation}
and consequently
\begin{equation} 
\mathcal{H}_\mu=-X_{1}X_{1}\delta-X_{2}X_{2}\delta-(c_{12}^{2}-c^\theta_{\theta1})\left(X_{1}\delta\right)+(c_{12}^{1}+c^\theta_{\theta2})\left(X_{2}\delta\right).
\label{eq:Hmu2}
\end{equation}
\end{remark}
\medskip
Let us consider a smooth measure $\mu$ on $M$ which is the Riemannian volume with respect to a Riemannian metric $g$ that extends the sub-Riemannian metric.
 Therefore, an orthonormal frame for $g$ is of the form $\{X_1,X_2,X_\theta\}$ as in Remark~\ref{rem:Hmu}. 

One can define a corresponding Riemannian approximation of the contact sub-Riemannian manifold as $(M,g^\eps)$, where $g^\eps$ is the Riemannian metric for which $\{X_1,X_2,\eps X_\theta\}$ is an orthonormal frame.

The next proposition states that the sub-Riemannian mean curvature with respect to $\mu$ is the limit of the corresponding weighted Riemannian mean curvature. 
\begin{proposition}
\label{prop:Hriemannapproximation} Let $\mu$ be a smooth measure which is the Riemannian volume with respect to a metric $g$ that extends the sub-Riemannian metric.
The mean sub-Riemannian curvature $\mathcal{H}_\mu$ of a regular surface $S$ in $M$ is the limit of the mean curvatures $H^\eps_{\mu}$ of $S$ in the corresponding Riemannian approximations  
\begin{equation}
\mathcal{H}_\mu=\lim_{\eps\to 0}H_\mu^\eps.
\end{equation}
\end{proposition}
 For the reader's convenience, the proof is postponed in Appendix~\ref{a:mean}.

\subsection{Coarea formula} 

In order to compute the volume of $U_\eps$ we need a sub-Riemannian version of the coarea formula. We decompose the space $U_\eps=\bigcup_{t\in(0,\eps)} U^t$, where $U^t=\delta^{-1}(t)$ and we choose an adapted frame to such a decomposition.

\begin{definition}
\label{def: moving frame}
We define the orthogonal moving frame on $TU_{\eps}:$
\begin{align}
F_{1}&=\left(X_{2}\delta\right)X_{1}-\left(X_{1}\delta\right)X_{2},\\
F_{2}&=\left(X_{0}\delta\right)\nabla_{H}\delta-X_{0},\\
N&=\nabla_{H}\delta.
\end{align}
We notice that, for fixed $t\in (0,\eps)$, the vectors $F_{1},F_{2}$ define a basis of $TU^{t}$, while $N$ is the sub-Riemannian normal. Moreover, $F_1,F_2$ extend $X_{S}$ and $Y_{S}$ defined on $U$ given in \eqref{eq:XS} and \eqref{eq: YS} respectively.
The corresponding dual basis on $T^*U_\eps$ is:
\begin{align}
\phi^{1}&=\left(X_{2}\delta\right)\omega^{1}-\left(X_{1}\delta\right)\omega^{2},\\
\phi^{2}&=- \omega,\\
\eta&=\left(X_{1}\delta\right)\omega^{1}+\left(X_{2}\delta\right)\omega^{2}+\left(X_{0}\delta\right)\omega.
\end{align}
\end{definition}

Let us now recall the classical Riemannian coarea formula for smooth functions (see for instance \cite{chavel}).
\begin{proposition}[Coarea formula] 
\label{prop: coareaR}
Let $(M,g)$ be a Riemannian manifold equipped with a smooth measure $\mu$. Let $\Phi:M\to\mathbb{R}$ be a smooth function such that $\nabla\Phi\neq 0$ and $f:M\to\mathbb{R}$ be a measurable non negative function, then 
\begin{equation}
\int_{M}f\left\|\nabla\Phi\right\|dV_{\mu}=\int_{\R}\int_{\Phi^{-1}(t)}f dA^R_{t}dt,
\end{equation}
where $dV_{\mu}$ and $dA^R_{t}$ are respectively the volume element on $M$ induced by $\mu$, and the induced Riemannian area form on $\Phi^{-1}(t)$. 
\end{proposition}

We recall that the induced area form on $\Phi^{-1}(t)$, for $t\in \R$, is given by the interior product of the volume form $dV_\mu$ by $\frac{\nabla\Phi}{\left|\nabla\Phi\right|}$, the Riemannian normal to the surface, i.e.,
\begin{equation}
dA^R_{t}=\iota\left(\frac{\nabla\Phi}{\left\|\nabla\Phi\right\|}\right)dV_\mu\mid_{ T\Phi^{-1}(t)}.
\end{equation}  
Now, let us consider $\left(M,\distr,g\right)$ a three-dimensional contact manifold equipped with a smooth measure $\mu$. Consider $\Phi:M\to\mathbb{R}$ a smooth function such that $\nabla_H\Phi\neq 0$, we define the induced sub-Riemannian area form on the surfaces $\Phi^{-1}(t)$
\begin{equation}
dA_{t}=\iota\left(\frac{\nabla_H\Phi}{\left\|\nabla_H\Phi\right\|}\right)dV_{\mu\mid T\Phi^{-1}(t)}.
\end{equation}  
This is the restriction of the volume form $dV_\mu$ on $M$, induced by the measure $\mu$, by the sub-Riemannian normal $\frac{\nabla_H\Phi}{\left\|\nabla_H\Phi\right\|}$ to the surface $\Phi^{-1}(t)$. 

We are ready to reformulate the coarea formula as follows.

\begin{proposition}[Sub-Riemannian Coarea] \label{prop:coareaSR}
Let $\left(M,\distr,g\right)$ be a three-dimensional contact manifold equipped with a smooth measure $\mu$. Given $\Phi:M\to\mathbb{R}$ a smooth function such that $\left\|\nabla_H\Phi\right\|\neq0$ and   $f:M\to\mathbb{R}$ measurable non negative
\begin{equation}
\int_{M}f\left\|\nabla_H\Phi\right\|d\mu=\int_{\mathbb{R}}\int_{\Phi^{-1}(t)}f_{\mid\Phi^{-1}(t)}dA_{t}dt,
\end{equation}
where $dA_t$ is the induced sub-Riemannian area form on $\Phi^{-1}(t)$. 
\end{proposition}
Turning the attention to our problem and taking in account the bases previously defined on $U_\eps$ (Definition~\ref{def: moving frame}), we have that the volume form $dV_{\mu}$ coincides with $\eta\wedge\phi^{1}\wedge\phi^{2}$ up to a smooth density $h:M\to\R$. Then, the sub-Riemannian area form induced by $\mu$ on $U^{t}$ is
\begin{equation}
\label{eq:SRarea}
dA_{t}=\iota\left(N\right)dV_{\mu}=h\cdot\iota\left(N\right)\eta\wedge\phi^{1}\wedge\phi^{2}=h\cdot\phi^{1}\wedge\phi^{2}.
\end{equation}
Finally, applying the sub-Riemannian coarea formula, and from \eqref{eq:gradH=1},
\begin{equation}
\label{eq:misuraUe}
\mu(U_\eps)=\int_{U_\eps}\frac{\|\nabla_H\delta\|}{\|\nabla_H\delta\|}dV_\mu=\int_0^\eps\int_{U^t} dA_tdt.
\end{equation}
\begin{proof}
To obtain the result we have to clarify the relation between the induced Riemannian area form and the sub-Riemannian one on  $\Phi^{-1}(t)$.
As noticed before, the measure $\mu$ is proportional to the measure associated with Popp's volume up to a smooth density, and the same relation is true for the associated measures, volumes and area forms.
Therefore, without lack of generality, we suppose that $dV_\mu$ is the Popp's volume $\omega\wedge\omega^1\wedge\omega^2$.
First of all, let us consider the following moving frame of vector fields that coincides with that of Definition~\ref{def: moving frame} when $\Phi=\delta$:
\begin{gather}
F_{1}=\left\|\nabla_{H}\Phi\right\|^{-2}\left[\left(X_{2}\Phi\right)X_{1}-\left(X_{1}\Phi\right)X_{2}\right],\\
F_{2}=\left(X_{0}\Phi\right)\nabla_{H}\Phi-\left\|\nabla_{H}\Phi\right\|^2X_{0},\\
N=\left\|\nabla_{H}\Phi\right\|^{-2}\nabla_{H}\Phi.
\end{gather}
The corresponding dual basis of 1-forms is written as follows:
\begin{gather}
\phi^{1}=\left(X_{2}\Phi\right)\omega^{1}-\left(X_{1}\Phi\right)\omega^{2},\\
\phi^{2}=-\left\|\nabla_{H}\Phi\right\|^{-2}\omega,\\
\eta=\left(X_{1}\Phi\right)\omega^{1}+\left(X_{2}\Phi\right)\omega^{2}+\left(X_{0}\Phi\right)\omega.
\end{gather}
For a fixed $t\in\mathbb{R}$, the vector fields $F_1$ and $F_2$ are tangent  to the level set $\Phi^{-1}(t)$, while $N$ is transversal. 
The 1-form $\eta$ annihilates the tangent space, while $\phi^1\wedge\phi^2$ is an area form on the surface.
Moreover, we have that
\begin{align}
\eta\wedge\phi^{1}\wedge\phi^{2}\left(X_{0},X_{1},X_{2}\right) & =\eta\left(X_{1}\right)\phi^{1}\left(X_{2}\right)\phi^{2}\left(X_{0}\right)-\eta\left(X_{2}\right)\phi^{1}\left(X_{1}\right)\phi^{2}\left(X_{0}\right)\\
 & =\left \|\nabla_{H}\Phi\right\|^{-2}\left[\left(X_{1}\Phi\right)^{2}+\left(X_{2}\Phi\right)^{2}\right]=1,
\end{align}
and this means that the volume form $\eta\wedge\phi^1\wedge\phi^2$ coincides with the volume form $\omega\wedge\mathrm{d}\omega$ associated with the $\Popp$ measure.
Now, let us consider the Riemannian and sub-Riemannian induced area forms on $\Phi^{-1}(t)$:
 \begin{eqnarray}
 dA^R_t=\iota\left(\frac{\nabla\Phi}{\left\|\nabla\Phi\right\|}\right)d\mu,& \mathrm{and} &
 dA_t=\iota\left(\frac{\nabla_{H}\Phi}{\left\|\nabla_{H}\Phi\right\|}\right)d\mu.
 \end{eqnarray}
To compute those area forms, we need to write the Riemannian
 outer normal in our new basis
\begin{equation}
\frac{\nabla\Phi}{\left\|\nabla\Phi\right\|}=\frac{\nabla_H\Phi+(X_0\Phi)X_0}{\left\|\nabla\Phi\right\|}=\rho N+AF_{1}+BF_{2}.
\end{equation}
We obtain that $A=0$ because $F_1$ is orthogonal to $\nabla\Phi$. 
Moreover, projecting on $X_0$, it holds that
$$B=-\dfrac{1}{\left\|\nabla_{H}\Phi\right\|^2}\dfrac{X_0\Phi}{\left\|\nabla\Phi\right\|}.$$ 
Then, substituting and computing the scalar product with $N$,
\begin{equation}
\dfrac{1}{\left\|\nabla\Phi\right\|}=\dfrac{\rho}{\left\|\nabla_{H}\Phi\right\|^2}-\dfrac{(X_0\Phi)^2}{\left\|\nabla_{H}\Phi\right\|^2\left\|\nabla\Phi\right\|}
\end{equation}
and we conclude that
\begin{equation}
\rho=\frac{\left\|\nabla_{H}\Phi\right\|^2}{\left\|\nabla\Phi\right\|}+\frac{\left(X_0\Phi\right)^2}{\left\|\nabla\Phi\right\|}=\left\|\nabla\Phi\right\|.
\end{equation}
Since we consider $2$-forms on $T\{\Phi^{-1}(t)\}$, we deduce that 
\begin{gather}
dA^R_t=\iota \left(\dfrac{\nabla\Phi}{\|\nabla\Phi\|}\right)dV_\mu=\iota \left(\rho N\right)dV_\mu=\left\|\nabla\Phi\right\|\phi^1\wedge\phi^2,\\
dA_t=\iota \left(\dfrac{\nabla_H\Phi}{\|\nabla_H\Phi\|}\right)dV_\mu=\iota \left(\left\|\nabla_H\Phi\right\| N\right)dV_\mu=\left\|\nabla_H\Phi\right\|\phi^1\wedge\phi^2.
\end{gather}
Finally, applying the coarea Riemannian formula in Proposition~\ref{prop: coareaR},
\begin{equation}
\int_Mf\left\|\nabla_H\Phi\right\|dV_\mu=\int_\mathbb{R}\int_{\{\Phi=t\}}f\frac{\left\|\nabla_H\Phi\right\|}{\left\|\nabla\Phi\right\|}dA^R_tdt=\int_\mathbb{R}\int_{\{\Phi=t\}}fdA_tdt. \qedhere
\end{equation}
\end{proof}

\section{A localized formula for the volume of the half-tube}
\label{s:localized}

The measure of the localized half-tubular neighborhood $U_\eps$ is expressed in  \eqref{eq:misuraUe} as
\begin{equation}
\mu\left(U_{\eps}\right)=\int_{0}^{\eps}\int_{U^{t}}dA_{t}dt,
\end{equation}
where, for $t\in (0,\eps)$, $dA_t$ is the sub-Riemannian area form \eqref{eq:SRarea} induced by the measure $\mu$ on the surface $U^t=G(t,U)$, with $G:(0,\eps)\times U\to U_\eps$ the diffeomorphism of Theorem~\ref{thm:smoothness}. 
One can further transform the previous formula considering the diffeomorphism $G$ at fixed time. Namely, the sub-Riemannian area form is 
\begin{equation}
dA_t\mid_{(t,p)}=h(t,p)\cdot \phi^1\wedge\phi^2\mid_{(t,p)}=h(t,p)\cdot  G_*\left(\phi^1\wedge\phi^2\mid_{ (0,p)}\right),
\end{equation}
and the expression for the measure of $U_\eps$ becomes the following
\begin{equation}
\label{eq: mu(Ue)}
\mu\left(U_{\eps}\right)=\int_{0}^{\eps}\int_{U}\left|\text{det}\left(d_{p}G\vert _{(t,p)}\right)\right|\frac{h(t,p)}{h(0,p)} dA_\mu dt,
\end{equation}
where with $dA_\mu$ we denote $dA_0$, that is the induced sub-Riemannian area form on $U$ by the measure $\mu$ in \eqref{eq:SRarea}.
We deduce that $\mu(U_\eps)$ is smooth at $\eps=0$ and, in order to obtain its asymptotics stated in Theorem~\ref{thm:sviluppo}, we need to compute the Taylor expansion near $t=0$ of the function
\begin{equation}
\label{eq: det}
t\mapsto\left|\text{det}\left(d_{p}G\vert _{(t,p)}\right)\right|h(t,p),
\end{equation} 
for a fixed $p\in U$. 
For this purpose, we express the matrix representing $d_{p}G\vert _{(t,p)}:T_pU\to T_{G(t,p)}U^t$ with respect to some bases as follows: on $T_{G(t,p)}U^t$, we choose $\{F_1,F_2\}|_{U^t}$ of the moving frame of Definition~\ref{def: moving frame}; on $T_pU$ we consider $\{X_S(p),Y_S(p)\}$ introduced in \eqref{eq:XS} and \eqref{eq: YS}. 
Moreover,  we define following vector fields  on $U_\eps$:
\begin{eqnarray}
\label{eq:Vi}
V_{1}(t,p)=dG\left(X_S(p)\right), &V_2(t,p)=dG\left(Y_{S}(p)\right),
\end{eqnarray}
where in the left hand side $(t,p)$ is a shorthand for computing the vector field at $G(t,p)\in U_\eps$.

\begin{proposition}
\label{prop:determinant}
Let $(t,p)\in (0,\eps)\times U$, the matrix representing the linear operator $d_{p}G\vert _{(t,p)}:T_pU\to T_{G(t,p)}U^t$ with respect to the bases $X_S,Y_S$ for $T_pU$ and $F_1,F_2$ for $T_{G(t,p)}U^t$, is
\begin{equation}
\label{eq:matrixdeterminant}
d_{p}G\vert _{(t,p)}=\left(\begin{array}{cc}
-\left\langle V_{1},JN\right\rangle  & -\left\langle V_{2},JN\right\rangle \\
-\left\langle V_{1},X_{0}\right\rangle  & -\left\langle V_{2},X_{0}\right\rangle 
\end{array}\right)(t,p),
\end{equation}
where $N=\nabla_H\delta$, $J$ is the linear map of Definition~\ref{def:J} and $X_0$ is the Reeb vector field of the contact manifold.
 In particular, it holds that 
\begin{equation}
\label{eq:detci}
\mu(U_\eps)=\int_U\int_0^\eps\Big|\left\langle V_1,JN\right\rangle\left\langle V_2,X_0\right\rangle -\left\langle V_2,JN\right\rangle\left\langle V_1,X_0\right\rangle \Big|(t,p)\frac{h(t,p)}{h(0,t)}dtdA_\mu.
\end{equation}
\end{proposition}

Before giving the proof we compare the vector fields $V_1,V_2,F_1,F_2$, and present a property that is useful also for the incoming computations.

By construction, for every $(t,p)\in U_\eps$, the vector fields $V_1$, $V_2$ are tangent to $U^t$. As already mentioned, the same is true for $F_1,F_2$, however they differ from the couple $V_1$, $V_2$ just introduced,
as the following lemma states.

\begin{lemma}
Let us consider the vector fields $F_1,N$ defined in Definition~\ref{def: moving frame}, and $V_1$, $V_2$ introduced in \eqref{eq:Vi}. For every $(t,p)\in U_\eps$, we have $\left[F_1,N\right]\neq 0$ and 
\begin{equation}
\label{eq:[Vi,N]=0}
\left[V_{1},N\right]=\left[V_{2},N\right]=0.
\end{equation}
\end{lemma}

\begin{proof}
Since $F_1$ and $N$ are independent horizontal vector fields in a three-dimensional contact sub-Riemannian manifold, the H\"ormander condition implies that the Lie bracket $\left[F_1,N\right]$ has to be everywhere nonvanishing.

Now, let us choose in $U$ some coordinates $\left( u_1,u_2\right)$ with corresponding coordinate fields $\left\{\partial_{u_1},\partial_{u_2}\right\}$.
The vector fields $X_S,Y_S$ are tangent to $U$ and can be written as
\begin{eqnarray}
X_S=a_X^1\partial_{u_1}+a_X^2\partial_{u_2},
& Y_S=a_Y^1\partial_{u_1}+a_Y^2\partial_{u_2}.
\end{eqnarray} 
with $a_X^1,a_X^2,a_Y^1,a_Y^2$ smooth functions on $U$.
Moreover, since $G:(0,\eps)\times U\to U_\eps$ a diffeomorphism, we have that, for $i=1,2$, $$\left[dG\left(\partial_t\right),dG\left(\partial_{u_1}\right)\right]=dG\left[\partial_t,\partial_{u_1}\right]=0.$$
Finally, by Theorem~\ref{thm:smoothness}, the horizontal gradient of the sub-Riemannian distance $N=\nabla_H\delta$ is $dG(\partial_t)$. Thus, it holds that
\begin{equation}
\left[V_1,N\right]=dG\left[a_X^1\partial_{u_1}+a_X^2\partial_{u_2},\partial_t\right]=-\partial_t(a_X^1)\partial_{u_1}-\partial_t(a_X^2)\partial_{u_2}=0.
\end{equation}
 An analogous computation proves that $\left[V_2,N\right]=0$.
\end{proof}

\begin{proof}[Proof of Proposition~\ref{prop:determinant}]
To write the matrix that represents $d_{p}G\vert _{(t,p)}$, we need to project
the vector fields $V_1(t,p)$ and $V_2(t,p)$ along the orthonormal basis $\frac{F_1}{\|F_1\|},\frac{F_2}{\|F_2\|}$ in $T_{G(t,p)}U^t$. Then the linear combinations we are looking for are the following
\begin{equation}
V_i=\frac{ \left\langle V_{i},F_1\right\rangle}{\|F_1\|^2}F_1+\frac{ \left\langle V_{i},F_2\right\rangle}{\|F_2\|^2}F_2,\qquad i=1,2,
\end{equation}
and the matrix of our interest is
\begin{equation}
d_{p}G\vert _{(t,p)}=\left(\begin{array}{cc}
\dfrac{ \left\langle V_{1},F_1\right\rangle}{\|F_1\|^2}  & \dfrac{ \left\langle V_{2},F_1\right\rangle}{\|F_1\|^2} \\[0.4cm]
\dfrac{ \left\langle V_{1},F_2\right\rangle}{\|F_2\|^2}  & \dfrac{ \left\langle V_{2},F_2\right\rangle}{\|F_2\|^2} 
\end{array}\right)(t,p).
\end{equation}
Now, recall that $N$ is horizontal and unitary, so that, considering in addition $JN$ and $X_0$ we have an orthonormal frame for $TU_\eps$. 
Therefore, looking at the formula in Definition~\ref{def: moving frame}, we have that in the new basis
\begin{eqnarray}
F_1=-JN & \mbox{and} & F_2=(X_0\delta)N-X_0.
\end{eqnarray}
Hence $\|F_1\|^2=1$ and $\|F_2\|^2=1+(X_0)^2$.
Moreover, due to the fact that the Riemannian gradient of $\delta$, i.e., $\nabla\delta=N+(X_0\delta)X_{0}$, is orthogonal to every surface $U^t=\delta^{-1}(t)$, we have that, for $i=1,2$, $\left\langle V_{i},N\right\rangle =-(X_0\delta)\left\langle V_{i},X_{0}\right\rangle .$
Therefore,
\begin{equation}
\left\langle V_{i},F_2\right\rangle 
 =(X_0\delta)\left\langle V_{i},N\right\rangle -\left\langle V_{i},X_{0}\right\rangle 
 =-\left(1+(X_0\delta)^2\right)\left\langle V_{i},X_{0}\right\rangle. 
\end{equation}
Finally, substituting we conclude the proof. 
\end{proof}

\section{The Taylor expansion. Proof of Theorem~\ref{thm:sviluppo}}
\label{s:expansion}

Let $S$ be a smooth surface bounding a closed region in a three-dimensional sub-Riemannian contact manifold equipped with a smooth measure $\mu$. Let us consider $U\subset S$ an open relatively compact set such that its closure does not contain any characteristic point. 
We recall that \eqref{eq:detci}, in Proposition~\ref{prop:determinant}, expresses the volume of the localized half-tubular neighborhood  $U_\eps$ with respect to $\mu$:
\begin{equation}
\mu(U_\eps)=\int_U\int_0^\eps\Big|\left\langle V_1,JN\right\rangle\left\langle V_2,X_0\right\rangle -\left\langle V_2,JN\right\rangle\left\langle V_1,X_0\right\rangle \Big|(t,p)\frac{h(t,p)}{h(0,t)}dtdA_\mu,
\end{equation}
where $h:U_\eps\to\R$ is the smooth density of $\mu$ with respect the Popp's measure $\Popp$ (i.e., $\mu=h\Popp$), $dA_\mu$ is the induced sub-Riemannian area form on $U$ by $\mu$ (i.e., $dA_0$ in \eqref{eq:SRarea}), $V_1,V_2$ are the vector fields defined in \eqref{eq:Vi}, and $N=\nabla_H\delta$ is the sub-Riemannian gradient of the sub-Riemannian distance from the surface.

In order to prove the statement in Theorem~\ref{thm:sviluppo}, we fix $p\in U$, and we proceed in calculating the Taylor expansion centered in $t=0$ of the smooth function
\begin{equation}
\label{eq:CC}
C(t,p)=| \left\langle V_1,JN\right\rangle\left\langle V_2,X_0\right\rangle -\left\langle V_2,JN\right\rangle\left\langle V_1,X_0\right\rangle |(t,p),
\end{equation}
Therefore, for $ k=1,2,3$ the coefficients of the expansion \eqref{eq:expansion} are given by
\begin{equation}
a_k=\int_U\sum_{\substack{ i,j=0,\dots,k-1\\ i+j=k-1}} \binom{k-1}{i}\partial^{i}_tC\cdot\frac{ \partial^j_t h}{h} (0,p) dA_\mu,
\end{equation} 
from which we obtain the expressions in \eqref{eq:coefficienti} thanks to \eqref{eq:Hm}.

We start the proof observing a useful relation between the terms appearing in \eqref{eq:CC}. Recall that $N=\partial_t$ thanks to Theorem~\ref{thm:smoothness}.
\begin{lemma}
\label{lem:N(V,X0)=dt(V,JN)}
Let us consider the vector fields $V_{i}$ defined in \eqref{eq:Vi}, for $i=1,2$. On the set $U_\eps$ it holds that
\begin{equation}
\partial_t \left\langle V_{i},X_{0}\right\rangle  =N\left\langle V_{i},X_{0}\right\rangle =\left\langle V_{i},JN\right\rangle.
\end{equation}
\end{lemma}

\begin{proof}
We employ the properties of the Tanno connection introduced in Definition~\ref{def:Tanno}.
Recalling property \eqref{eq:[Vi,N]=0} and that $\nabla_XN$ is parallel to $JN$ for every vector field $X$ since $N$ is unitary and horizontal, we deduce that
\begin{equation}
N\left\langle V_{i},X_{0}\right\rangle 
=\left\langle \nabla_N V_i, X_0\right\rangle
=\left\langle \nabla_{V_i}N+[N,V_i]+\Tor(N,V_i), X_0\right\rangle
=\left\langle \Tor(N,V_i), X_0\right\rangle.
\end{equation} 
We split $V_{i}$ in its the horizontal and Reeb component to exploit the properties of the torsion:
\begin{align}
N\left\langle V_{i},X_{0}\right\rangle &=\left\langle \Tor(N,V_i-\langle V_i,X_0\rangle X_0)+\Tor(N,\langle V_i,X_0\rangle X_0), X_0\right\rangle\\
&=-\left\langle N,JV_i-\langle V_i,X_0\rangle JX_0\right\rangle
\end{align}
and we conclude since $JX_0=0$.
\end{proof}

Set $c_i(t,p)= \left\langle V_i,X_{0}\right\rangle.$
Thanks to Lemma~\ref{lem:N(V,X0)=dt(V,JN)}, formula \eqref{eq:CC} can be rewritten as
\begin{equation}
C(t,p)=|\dot{c_{1}}c_{2}-c_{1}\dot{c_{2}}|(t,p),
\label{eq: detconci}
\end{equation}
We recall that $C(t,p)$ comes from the determinant of matrix \eqref{eq:matrixdeterminant}, that for every $p\in U$ at $t=0$ has the following expression:
\begin{equation}
\label{eq: matrixci}
\left(\begin{array}{cc}
-\dot{c_{1}}(0) & -\dot{c_{2}}(0)\\
-c_{1}(0) & -c_{2}(0)
\end{array}\right)
=\left(\begin{array}{cc}
-\left\langle V_{1},JN\right\rangle  & -\left\langle V_{2},JN\right\rangle \\
-\left\langle V_{1},X_{0}\right\rangle  & -\left\langle V_{2},X_{0}\right\rangle 
\end{array}\right)(0,p)
=\left(\begin{array}{cc}
1 & 0\\
0 & 1
\end{array}\right).
\end{equation}
Since $C(0,p)=\left|\dot{c_1}c_2-c_1\dot{c_2}\right|(0,p)=1$, we deduce the first term $a_1$ of the expansion \eqref{eq:expansion} in Theorem~\ref{thm:sviluppo}. Namely
\begin{equation}
a_1=\int_UC(0,p)dA_\mu=\int_UdA_\mu.
\end{equation} 
Again, taking into account the identities in \eqref{eq: matrixci}, we obtain the following expressions for the derivatives of $C$ in $t=0$:
\begin{equation}\label{eq:CCprimo0}
\partial_t C(0,p)=\left(\ddot{c}_{1}c_{2}-c_{1}\ddot{c}_{2}\right)(0,p)=-\ddot{c}_{1}(0,p);
\end{equation}
\begin{equation}
\partial^2_t C(0,p)=\left(c^{(3)}_{1}c_{2}+\ddot{c}_{1}\dot{c}_{2}-\dot{c}_{1}\ddot{c}_{2}-c_{1}c^{(3)}_{2}\right)(0,p)=-c^{(3)}_{1}(0,p)+\ddot{c_{2}}(0,p).\label{eq:CCsecondo0}
\end{equation}
The work now is about to explicit these formulae with respect the sub-Riemannian curvature operators.
In order to do that, we show the relation between the Lie brackets and the canonical Tanno connection when considering the key vector fields for the surface $N$ and $JN$.

\begin{lemma}
\label{lem:b0} We have that at every point of $U_\eps$
\begin{equation}
 \left[JN,N\right]=-\mathcal{H}JN-\left(X_0\delta\right)N+X_{0},
\end{equation} 
where $\mathcal{H}:U_\eps\to \R$ is the sub-Riemannian mean curvature of the surface with respect to the Popp measure $\Popp$ in Definition~\ref{def: H}.
Moreover,
\begin{eqnarray}
\nabla_{JN}N=-\mathcal{H}JN & \mbox{ and }& \nabla_NN=-(X_0\delta) JN.
\end{eqnarray}
\end{lemma}

\begin{proof}
Exploiting \eqref{eq:gradH=1}, we can express the vector fields $N$ and $JN$ as a rotation of $X_1,X_2$: 
\begin{equation}
\left(\begin{array}{c}
N\\
JN
\end{array}\right)=\left(\begin{array}{cc}
X_{1}\delta & X_{2}\delta\\
-X_{2}\delta & X_{1}\delta
\end{array}\right)\left(\begin{array}{c}
X_{1}\\
X_{2}
\end{array}\right)=\left(\begin{array}{cc}
\text{\ensuremath{\cos\theta}} & -\sin\theta\\
\sin\theta & \cos\theta
\end{array}\right)\left(\begin{array}{c}
X_{1}\\
X_{2}
\end{array}\right).
\end{equation}
Therefore, by Lemma~\ref{lem:rotation},
\begin{align}
\left\langle \left[JN,N\right],JN\right\rangle 
 & =\left\langle \left[X_{2},X_{1}\right]-X_{1}\left(\theta\right)X_{1}-X_{2}\left(\theta\right)X_{2},\sin\theta X_{1}+\text{\ensuremath{\cos\theta}}X_{2}\right\rangle \\
 & =\sin\theta c_{12}^{1}+\text{\ensuremath{\cos\theta}}c_{12}^{2}-\sin\theta X_{1}\left(\theta\right)-\text{\ensuremath{\cos\theta}}X_{2}\left(\theta\right)\\
 & =-\left(X_{2}\delta\right)c_{12}^{1}+\left(X_{1}\delta\right)c_{12}^{2}+X_{1}X_{1}\delta+X_{2}X_{2}\delta,
\end{align}
that is exactly the expression of $-\mathcal{H}$ in coordinates, see Lemma~\ref{lem: H}.
With similar computations we obtain that
\begin{align}
\left\langle \left[JN,N\right],N\right\rangle 
 & =\left\langle \left[X_{2},X_{1}\right]-X_{1}\left(\theta\right)X_{1}-X_{2}\left(\theta\right)X_{2},\cos\theta X_{1}-\sin\theta X_{2}\right\rangle \\
 & =\left(X_{1}\delta\right)c_{12}^{1}+\left(X_{2}\delta\right)c_{12}^{2}+X_{1}X_{2}\delta-X_{2}X_{1}\delta\\
 &=\left(X_{1}\delta\right)c_{12}^{1}+\left(X_{2}\delta\right)c_{12}^{2}-[X_{2},X_{1}]\delta=-(X_0\delta).
 \end{align}
 
 Finally, since $\Tor\left(JN,N\right)=-X_{0}$, we have that
\begin{equation}
\left[JN,N\right]  =\nabla_{JN}N-\nabla_{N}JN+X_{0},
\end{equation}
and thus we conclude from the fact that $N$ and $JN$ are unitary. In particular,
 \begin{equation}
 \nabla_{JN}N=-\mathcal{H}JN,\qquad  -\nabla_NJN=-J\nabla_NN=-(X_0\delta)N. \qedhere
 \end{equation}
\end{proof}

The next two propositions concern the computations of the elements to obtain the second and third term in the Taylor expansion of Theorem~\ref{thm:sviluppo}. 

\begin{proposition}
\label{prop: Cprimo0}
For every $p\in U$, it holds that 
\begin{equation}
\partial_t{C}(0,p)=\left\langle \left[JN,N\right],JN\right\rangle (0,p) =-\mathcal{H}(p)
\end{equation} 
where $\mathcal{H}:U_\eps\to\mathbb{R}$ is the mean sub-Riemannian curvature of the surface with respect to the Popp measure $\Popp$ defined in Definition~\ref{def: H}.
\end{proposition}

\begin{proof}
Considering equation \eqref{eq:CCprimo0} and recalling that the Tanno connection is metric, we have 
\begin{equation}
\partial_t{C}(t,p)=-\partial_t\left\langle V_1,JN\right\rangle =-\left\langle \nabla_{N}V_{1},JN\right\rangle -\left\langle V_{1},\nabla_{N}JN\right\rangle .
\end{equation}
Evaluating at $t=0$ the second addendum vanishes being $V_1(0,p)=-JN(0,p)$.
For what concerns the first addendum, the idea is to get rid of $V_1$ applying the definition of the torsion operator and property \eqref{eq:[Vi,N]=0}. Hence we have that
\begin{align}
-\left\langle \nabla_{N}V_{1},JN\right\rangle (0,p) & =-\left\langle \nabla_{V_{1}}N+\left[N,V_{1}\right]+\Tor\left(N,V_{1}\right),JN\right\rangle (0,p) \\
 & =\left\langle \nabla_{JN}N+\Tor\left(N,JN\right),JN\right\rangle (0,p).
\end{align}
We conclude thanks to Lemma~\ref{lem:b0} and the fact that $\Tor\left(N,JN\right)=X_0$.
\end{proof}

 \begin{proposition}
 Given $p\in U$, it holds that
 \begin{equation}
 \label{eq:lunga}
 \partial^2_t{C}(0,p)=\left[ X_{S}(X_0\delta)-\left(X_0\delta\right)^{2}-\left\langle R(JN,N)N,JN\right\rangle -\left\langle \nabla_{X_{0}}N,JN\right\rangle \right](0,p).
 \end{equation}
 Moreover, considering the mean curvature $\mathcal{H}:U_\eps\to \R$ of Definition~\ref{def: H}, the previous formula becomes 
 \begin{equation}
 \label{eq:CCorta}
 \partial_t^2C(0,p)=\mathcal{H}^{2}-N(\mathcal{H}).
\end{equation}  
 \end{proposition}

\begin{proof}
To obtain \eqref{eq:lunga} we compute $c^{(3)}_{1}(0,p)$ and $\ddot{c}_{2}(0,p)$ separately, being $\ddot{C}(0,p)$ expressed as their difference in \eqref{eq:CCsecondo0}.

\textit{Step 1.} We compute $c^{(3)}_{1}(0,p)$. 
Recall that $\dot{c}_1(t,p)=\left\langle V_1,JN\right\rangle (t,p)$, consequently:
\begin{equation}
c^{(3)}_{1}(t,p) =N^{2}\left\langle V_{1},JN\right\rangle = \left\langle \nabla_{N}\nabla_{N}V_{1},JN\right\rangle +2\left\langle \nabla_{N}V_{1},\nabla_{N}JN\right\rangle +\left\langle V_{1},\nabla_{N}\nabla_{N}JN\right\rangle.
\end{equation}
Computing the second and the third addendum in $t=0$, we have respectively:
\begin{align}
\left\langle \nabla_{N}V_{1},\nabla_{N}JN\right\rangle & =\left\langle \nabla_{V_{1}}N+\left[N,V_{1}\right]+\Tor\left(N,V_{1}\right),\nabla_{N}JN\right\rangle \\
 & =\left\langle -\nabla_{JN}N+\Tor\left(N,-JN\right),\nabla_{N}JN\right\rangle \\
 & =\left\langle -\nabla_{JN}N-X_0,\nabla_{N}JN\right\rangle =0,
\end{align}
because $\nabla_{JN}N,\nabla_{N}JN$ are orthogonal to each other, and then
\begin{align}
\left\langle V_{1},\nabla_{N}\nabla_{N}JN\right\rangle  & =-\left\langle JN,\nabla_{N}\nabla_{N}JN\right\rangle \\
 & =-N\left\langle JN,\nabla_{N}JN\right\rangle +\left\langle \nabla_{N}JN,\nabla_{N}JN\right\rangle \\
 & =\left\langle \nabla_{N}N,\nabla_{N}N\right\rangle \\
 & =\left(X_0\delta\right)^2 .
\end{align}
Therefore, $c^{(3)}_{1}(0,p)=\left[ \left\langle \nabla_{N}\nabla_{N}V_{1},JN\right\rangle +\left(X_0\delta\right)^2\right] (0,p)$.

\textit{Step 2.} We compute $\ddot{c}_{2}(t,p)$. Since $\dot{c}_2(t,p)=\left\langle V_2,JN \right\rangle (t,p)$, it holds that
\begin{equation}
\ddot{c}_{2}(t,p)=N\left\langle V_{2},JN\right\rangle   = \left\langle \nabla_{N}V_{2},JN\right\rangle +\left\langle V_{2},\nabla_{N}JN\right\rangle .
\end{equation}
Evaluating each addendum in $t=0$, where $V_2(0,p)=\left( X_0 \delta \right)N-X_0$, applying \eqref{eq:[Vi,N]=0} and Lemma~\ref{lem:b0}, we deduce that
\begin{align}
\left\langle \nabla_{N}V_{2},JN\right\rangle & =\left\langle \nabla_{V_{2}}N+\left[N,V_{2}\right]+\Tor\left(N,V_{2}\right),JN\right\rangle \\
 & =\left\langle \nabla_{\left( X_0 \delta \right)N-X_{0}}N+\Tor\left(N,\left( X_0 \delta \right)N-X_{0}\right),JN\right\rangle \\
 & =\left( X_0 \delta \right)\left\langle \nabla_{N}N,JN\right\rangle -\left\langle \nabla_{X_{0}}N+\Tor\left(N,X_{0}\right),JN\right\rangle  \\
 & =-\left( X_0 \delta \right)^2 +\left\langle \left[N,X_{0}\right],JN\right\rangle 
\end{align}
 and 
\begin{equation}
\left\langle V_{2},\nabla_{N}JN\right\rangle  =\left\langle \left( X_0 \delta \right)N-X_{0},\nabla_{N}JN\right\rangle =\left( X_0 \delta \right)^2.
\end{equation}
Putting together, we have that  $\ddot{c}_{2}(0,p)=\left\langle \left[N,X_{0}\right],JN\right\rangle (0,p)$. 

\textit{Step 3.} At this point we have the following intermediate expression for $\partial^2_t{C}(0,p)$:
\begin{equation}
\partial^2_t{C}(0,p)=\left[-\left\langle \nabla_{N}\nabla_{N}V_{1},JN\right\rangle -\left( X_0 \delta \right)^2 +\left\langle \left[N,X_{0}\right],JN\right\rangle \right](0,p).
\end{equation}

Let us focus on the term still depending on the Jacobi field $V_{1}$. In order to get rid of $V_{1}$ we need to use the curvature operator associated with the connection $\nabla$, i.e., $R(X;Y)Z=\nabla_X\nabla_YZ-\nabla_Y\nabla_XZ-\nabla_{[X,Y]}Z$. Then, applying \eqref{eq:[Vi,N]=0} and Lemma~\ref{lem:b0}, we have:
\begin{align}
-\left\langle \nabla_{N}\nabla_{N}V_{1},JN\right\rangle &= -\left\langle \nabla_{N}\nabla_{V_{1}}N+\nabla_{N}\Tor\left(N,V_{1}\right),JN\right\rangle \\
&= -\left\langle \nabla_{V_{1}}\nabla_{N}N+R(N,V_{1})N + \nabla_{\left[N,V_{1}\right]}N+\nabla_{N}\Tor\left(N,V_{1}\right),JN\right\rangle \\
&= \left\langle V_{1}(X_0\delta)JN+R(V_{1},N)N ,JN\right\rangle-\left\langle \nabla_{N}\Tor\left(N,V_{1}\right),JN\right\rangle .
\end{align}
Since the first scalar product is tensorial in the arguments in which $V_{1}$ appears, evaluating in $t=0$ we have: 
\begin{equation}
-\left\langle \nabla_{N}\nabla_{N}V_{1},JN\right\rangle =-JN(X_0\delta)-\left\langle R(JN,N)N ,JN\right\rangle-\left\langle \nabla_{N}\Tor\left(N,V_{1}\right),JN\right\rangle.
\end{equation}
At this point, looking at the scalar product involving the torsion term, we decompose $V_{1}$ along the orthonormal frame $N$, $JN$, $X_0$, i.e.,
\begin{align}
\Tor\left(N,V_{1}\right)&=\Tor\left(N,\langle V_{1},N\rangle N+\langle V_{1},JN\rangle JN+\langle V_{1},X_0\rangle X_0\right)\\
&=\langle V_{1},JN\rangle X_0 +\langle V_{1},X_0\rangle\Tor(N,X_0).
\end{align}
Evaluating at $t=0$, we get $\langle V_{1},X_0\rangle=0$ and we obtain that
\begin{align}
-\left\langle \nabla_{N}\Tor\left(N,V_{1}\right),JN \right\rangle 
 &= -\left\langle N\left\langle V_{1},JN\right\rangle X_{0}+N\left\langle V_{1},X_{0}\right\rangle \Tor\left(N,X_{0}\right),JN\right\rangle \\
 &= -\left\langle \left(V_{1},JN\right)\Tor\left(N,X_{0}\right),JN\right\rangle \\
 &= \left\langle \Tor\left(N,X_{0}\right),JN\right\rangle .
\end{align}
Since on $U$ we have that the characteristic vector field $X_S=-JN$, we get 
\begin{equation}
-\left\langle \nabla_N \nabla_NV_1,JN \right\rangle (0,p)=\left[X_S\left( X_0\delta\right) + \left\langle -R(JN,N)N+ \Tor\left(N,X_{0}\right),JN \right\rangle \right](0,p).
\end{equation}
Finally:
\begin{align}
\partial^2_t{C}(0,p)&=-\left\langle \nabla_{N}\nabla_{N}V_{1},JN\right\rangle -\left( X_0 \delta \right)^2 +\left\langle \left[N,X_{0}\right],JN\right\rangle \\
 &= X_{S}(X_0\delta)-\left(X_0\delta\right)^{2}+\left\langle -R(JN,N)N+\Tor\left(N,X_{0}\right)+\left[N,X_{0}\right],JN\right\rangle
\\ &= X_{S}(X_0\delta)-\left(X_0\delta\right)^{2}-\left\langle R(JN,N)N,JN\right\rangle -\left\langle \nabla_{X_{0}}N,JN\right\rangle .
\end{align}

\textit{Step 4.} The last part of the proof is to obtain \eqref{eq:CCorta}.
Exploiting Lemma~\ref{lem:b0}, we deduce that 
\begin{align}
-\left\langle R\left(JN,N\right)N,JN\right\rangle &= \left\langle -\nabla_{JN}\nabla_{N}N+\nabla_{N}\nabla_{JN}N+\nabla_{\left[JN,N\right]}N,JN\right\rangle \\
&= JN(X_0\delta) -N(\mathcal{H}) +\left\langle\nabla_{-\mathcal{H}JN-(X_0\delta)N+X_0}N,JN\right\rangle\\
&= -X_{S}\left(X_0\delta\right)-N\left(\mathcal{H}\right)+\mathcal{H}^{2}+\left(X_0\delta\right)^{2}+\left\langle \nabla_{X_{0}}N,JN\right\rangle.
\end{align}
Substituting in \eqref{eq:lunga}, we obtain the statement.
\end{proof}

\section{An operative formula. Proof of Proposition~\ref{prop:a3f}}
\label{s:operative}

In this Section we prove Proposition~\ref{prop:a3f}, where $\mu=\Popp$ is the Popp measure. In other words, we prove that one can compute the coefficients in expansion~\eqref{eq:expansion} of Theorem~\ref{thm:sviluppo}
only in terms of a smooth function $f$ locally defining the surface $S=f^{-1}(0)$, without the explicit knowledge of the sub-Riemannian distance function $\delta$. 

This property is true for the coefficients $a_{1}$ and $a_{2}$ appearing in the in expansion~\eqref{eq:expansion}, as it is evident from their explicit expression. Indeed, as noticed in Remark~\ref{rem:H}, the mean sub-Riemannian curvature $\mathcal{H}$, is well defined on the surface by \eqref{eq:HU}. This is not a priori clear for the coefficient $a_3$, since it depends on a derivative of $\mathcal{H}$ along $N$, which is a vector field transversal to the surface. Hence it might depend on the value of $\mathcal{H}$ outside $S$.

Before  proving Proposition~\ref{prop:a3f}, we recall the two local metric invariants $\chi$ and $\kappa$, that are smooth functions on $M$ characterizing a three-dimensional contact manifold.  
These quantities are strictly related to the curvature operator associated with the Tanno connection $\nabla$. 

%

The first metric functional invariant $\kappa$ of a three-dimensional sub-Riemannian contact manifold is the Tanno sectional curvature of the distribution, i.e., the smooth function defined as
 \begin{equation}
\kappa=\langle R(X_1,X_2)X_2,X_1\rangle,
\end{equation}
where $X_1,X_2$ is a local orthonormal frame for $\distr$.
Denoting by $\tau:\distr\to\distr$ the linear map $\tau(X)=\Tor(X_0,X)$, we define the second metric invariant $\chi$ as the function
\begin{equation}
\chi = \sqrt{-\det\tau}.
\end{equation}
In terms of \eqref{eq:structurecoeff} the metric invariants $\chi$ and $\kappa$ we have the following expression:
\begin{gather}
\chi^2=-c^1_{01}c^2_{02}+\frac{\left(c^2_{01}+c^1_{02}\right)^2}{4}, \\
\kappa=X_2\left(c^1_{12}\right)-X_1\left(c^2_{12}\right)-\left(c^1_{12}\right)^2-\left(c^2_{12}\right)^2+\frac{c^2_{01}-c^1_{02}}{2}.
\end{gather}
For more details see \cite[Sections~17.2,~17.6]{thebook} (see also \cite{Barilari2013}).

\begin{proof}[Proof of Proposition~\ref{prop:a3f}] 
As mentioned before, we need to prove the result only in the case of $a_3$.
Since both $f$ and $\delta$ locally define the surface, as their zero level set and $\langle\nabla_Hf,\nabla_H\delta\rangle|_U>0$, there exists a smooth function $h$ such that $f=e^h\delta$.  
Exploiting \eqref{eq:gradH=1}, we have that for $i=0,1,2$
\begin{equation}
\label{eq:Xf=Xd}
\left.\left(\dfrac{X_if}{\left\| \nabla_Hf\right\|}\right)\right|_U=\left.\left(\dfrac{X_i\delta+\delta X_ih}{\left\| \nabla_H\delta+\delta\nabla_Hh\right\|}\right)\right|_U=X_i\delta|_U.
\end{equation}
Therefore, the equivalence of the definition of $X_S$ in the current statement with that of \eqref{eq:XS} is proved.

Recalling \eqref{eq:lunga} we have the following expression for the coefficient $a_3$: 
\begin{equation}
a_3=\int_U\left[ X_{S}(X_0\delta)-\left(X_0\delta\right)^{2}-\langle R(JN,N)N,JN\rangle -\left\langle \nabla_{X_{0}}N,JN\right\rangle \right]dA.
 \end{equation}
Let us focus on the first term in the integral.
We observe that
\begin{align}
\nabla_H\left(\dfrac{X_0f}{\left\| \nabla_Hf\right\|} \right)&= \nabla_H\left(\dfrac{X_0\delta+\delta X_0h}{\left\|\nabla_H\delta+\delta\nabla_Hh\right\|}\right)\\
&=\frac{\nabla_H(X_0\delta)+(X_0h)\nabla_H\delta}{\left\|\nabla_H\delta+\delta\nabla_Hh\right\|}+(X_0\delta)\nabla_H\left(\dfrac{1}{\left\|\nabla_H\delta+\delta\nabla_Hh\right\|}\right).
\end{align}
And evaluating at $U\subset S$
\begin{align}
\nabla_H\left(\dfrac{X_0f}{\left\| \nabla_Hf\right\|} \right)
&=\nabla_H(X_0\delta)+(X_0h)N-(X_0\delta)\dfrac{1}{2}\nabla_H(1+2\delta Nh+\delta^2 \left\| \nabla_Hh\right\|^2)\\
&=\nabla_H(X_0\delta)+(X_0h-(X_0\delta)Nh)N.
\end{align}
Since $X_S$ is orthogonal to $N$, we deduce that on $U$
\begin{equation}
X_S(X_0\delta)=\langle-JN,\nabla_H\left(X_0\delta \right) \rangle = X_S\left(\dfrac{X_0f}{\left\| \nabla_Hf\right\|} \right).
\end{equation}
Moreover, since $N, JN$ is a local orthonormal frame for the sub-Riemannian structure, as well as $X_1,X_2$, a direct computation shows that 
\begin{equation}
\label{eq: kappaR}
\langle R(JN,N)N,JN\rangle=\langle R(X_1,X_2)X_2,X_1\rangle=\kappa. 
\end{equation}
Hence, from \eqref{eq:Xf=Xd}, we deduce that on $U$
\begin{equation}
-\left(X_0\delta\right)^{2}-\langle R(JN,N)N,JN\rangle = -\left(\dfrac{X_0f}{\left\| \nabla_Hf\right\|}\right)^2-\kappa.
\end{equation}
Finally, we observe that 
\begin{equation}
-\left\langle \nabla_{X_{0}}N,JN\right\rangle=\left\langle \Tor(N,X_0)+[N,X_0],JN\right\rangle.
\end{equation}
Then, making the computations for the second addendum we obtain that
\begin{align}
\left\langle [N,X_0],JN\right\rangle 
&= \left\langle [(X_1\delta)X_1+(X_2\delta)X_2,X_0],-(X_2\delta)X_1+(X_1\delta)X_2\right\rangle\\
&=(X_2\delta)(X_0X_1\delta)-(X_1\delta)(X_0X_2\delta)\\
& \quad +\left\langle (X_1\delta)[X_1,X_0] +(X_2\delta) [X_2,X_0],-(X_2\delta)X_1+(X_1\delta)X_2\right\rangle\\
&=(X_2\delta)(X_1X_0\delta)-(X_1\delta)(X_2X_0\delta)-c^2_{01}(X_2\delta)^2+c^1_{02}(X_1\delta)^2\\
&\quad+2(c^2_{02}-c^1_{01})(X_1\delta)(X_2\delta).
\end{align}
Recalling \eqref{eq:matricetau}, we deduce $\left\langle [N,X_0],JN\right\rangle =-JN(X_0\delta)+2\left\langle \Tor(X_0,N),JN\right\rangle,$ thus 
\begin{equation}
-\left\langle \nabla_{X_{0}}N,JN\right\rangle=X_S(X_0\delta)-\left\langle \Tor(X_0,N),X_S\right\rangle.
\end{equation}
The torsion operator is tensorial in its arguments, we conclude thanks to \eqref{eq:Xf=Xd}.
\end{proof}

\begin{remark}
If we specify Theorem~\ref{thm:sviluppo} to the three-dimensional Heisenberg group $\mathbb{H}$, we obtain a result that is consistent with \cite{ritore21}. 
Indeed, chosen an orthonormal frame $X_{1}$, $X_{2}$ for the distribution, we have the following relations:
\begin{eqnarray}
\left[X_2,X_1\right]=X_0& \text{and} &\left[X_1,X_0\right]=\left[X_2,X_{0}\right]=0.
\end{eqnarray}
The geometric invariant $\kappa$ vanishes and thus, exploiting Proposition~\ref{prop:a3f}, we compute the expansion \eqref{eq:expansion} with respect to $\Popp$:
\begin{equation}
\Popp(U_\eps)=\eps\int_UdA-\dfrac{\eps^2}{2}\int_U\mathcal{H}dA+\dfrac{\eps^3}{6}\int_U2X_{S}\left(X_0\delta\right)-\left(X_0\delta\right)^{2}dA+o(\eps^3).
\end{equation}

We warn the reader that in \cite{ritore21} the author chooses an orthonormal basis $X,Y,T$ for the Riemannian extension, that in our notations reads $X=X_1$, $Y=X_2$ and  $X_{0}=2T$. 
Moreover, 
 the vector field  $-e_1$ corresponds to $X_S$ in our notations, as well as $dP$ to $dA$, and $\frac{N}{\left\| N_h\right\|}$ to the Riemannian gradient $\nabla\delta=\nabla_H\delta+(X_0\delta)X_0$. 
 In particular, $2\frac{\left\langle N,T\right\rangle}{\left\| N_h\right\|}=X_{0}\delta$.  The third term in the expansion presented in \cite{ritore21}
\begin{equation}
-\frac{2\eps^{3}}{3}\int_{U}e_1\frac{\left\langle N,T\right\rangle}{\left\| N_h\right\|} +\frac{\left\langle N,T\right\rangle^2}{\left\| N_h\right\|^2} dP,
\end{equation}
coincides with the one obtained here.
\end{remark}

\begin{remark} \label{r:commenta4}
The explicit knowledge of the sub-Riemannian distance function $\delta$ from the surface $S$, is in general necessary in order to compute expansion~\eqref{eq:expansion} with order greater or equal than four.
Indeed, this is true already for surfaces $S$ in the Heisenberg group $\mathbb{H}$, thanks to the results in \cite{balogh15}.
Let  $U\subset S$ relatively compact and such that the closure does not contain any characteristic point. We have that $\Popp(U_\eps)=\sum_{k=1}^4\frac{\eps^k}{k!}a_k+o(\eps^4)$, with
\begin{equation}
\label{eq:a4}
a_4=\int_U(X_0\delta)^2\mathcal{H}+2X_0X_0\delta \,dA_{},
\end{equation}
where $\mathcal{H}=-X_1X_1\delta-X_2X_2\delta$ is the mean sub-Riemannian curvature of Definition~\ref{def: H} specified in the case of $\mathbb{H}$.

Let us consider a smooth local defining function $f:\mathbb{H}\to \R$ such that $U=f^{-1}(0)$ and $f=e^h\delta$ on $U_\eps$. The first term in \eqref{eq:a4} can be equivalently written as: 
\begin{equation}
(X_0\delta)^2\mathcal{H}|_U=-\left(\frac{X_0f}{\|\nabla_Hf\|}\right)^2\text{div}_{}\left(\frac{\nabla_Hf}{\|\nabla_Hf\|}\right).
\end{equation}
Moreover, since
\begin{equation}
X_0X_0f=e^h\left(\delta X_0X_0h+\delta (X_0h)^2+2(X_0h)(X_0\delta)+X_0X_0\delta\right),
\end{equation}
and $\|\nabla_Hf\| \big|_U=e^h\|\nabla_H\delta+\delta\nabla_Hh\| \big|_U=e^h$, we obtain that on the surface
\begin{equation}
X_0X_0\delta=\frac{X_0X_0f}{\|\nabla_Hf\|}-2(X_0h)\frac{X_0f}{\|\nabla_Hf\|}.
\end{equation}
The last identity shows that, in order to compute $a_4$, the explicit knowledge of the distance function $\delta$ is needed (i.e., one cannot get rid of $h$).
\end{remark}

We are ready to specialize our result for some particular three-dimensional  sub-Riemannian contact manifolds. 
Firstly, we consider the class of three-dimensional contact manifold with $\chi=0$. We recall that for this type of spaces there exists a local orthonormal frame such that 
\begin{equation}
\label{eq:bracketsk}
\begin{array}{ccc}
\left[X_{1},X_{0}\right]=\kappa X_{2}, & \left[X_2,X_0\right]=-\kappa X_1, & \left[X_{2},X_{1}\right]=X_{0}.
\end{array}
\end{equation}
In addition, when the structure of the contact manifold is left invariant on a Lie group with $\chi\neq0$, we can choose a canonical frame $X_{1}$, $X_{2}$ of the distribution that is unique up to a sign and such that:
\begin{equation}
\label{eq:canonicalbasis}
\begin{array}{ccc}
\left[X_{1},X_{0}\right]=c^2_{01}X_{2}, & \left[X_{2},X_{0}\right]=c^1_{02}X_{1}, & \left[X_{2},X_{1}\right]=X_{0}+c^1_{12}X_1+c^2_{12}X_2,\end{array}
\end{equation}
and with   (cf.\ \cite[Prop. 17.14]{thebook})
\begin{equation}
\chi=\frac{c^2_{01}+c^1_{02}}{2},
\end{equation} 
Taking into account these spaces and exploiting Proposition~\ref{prop:a3f} we have the following expression for the term $a_3$ in Theorem~\ref{thm:sviluppo}.
\begin{corollary}
\label{cor:corollary}
Given a three-dimensional contact manifold with $\chi=0$,
\begin{equation}
\label{eq:corollary}
a_3=\int_U2X_{S}(X_0\delta)-\left(X_0\delta\right)^{2}-\kappa\, dA.
\end{equation}
In the case of a left invariant sub-Rieman\-nian contact structure on a three-dimen\-sional Lie group with $\chi\neq0$, we have that
\begin{equation}
a_3=\int_U2X_{S}(X_0\delta)-(X_0\delta)^{2}-\kappa+\chi\left[\left(X_{1}\delta\right)^{2}-\left(X_{2}\delta\right)^{2}\right]dA
\end{equation}
with $X_1,X_2$ the unique (up to a sign) canonical basis chosen such that \eqref{eq:canonicalbasis} holds.
\end{corollary}

\begin{remark} 
The coefficients $a_1,a_2,a_3$ of expansion~\eqref{eq:expansion} in Theorem~\ref{thm:sviluppo} are integrals of iterated horizontal divergence of $\delta$. 
More precisely,
\begin{equation}
\label{eq:akdiv}
a_k=\int_U\mathrm{div}_\mu^{k-1}(\nabla_H\delta)\,dA_\mu\qquad k=1,2,3.
\end{equation}
The iterated divergences of a vector field $X$, with respect to a smooth measure $\mu$, on $M$ are defined as
\begin{eqnarray}
\text{div}^0_\mu (X)=1,& \text{div}^{k+1}_\mu (X)=\text{div}_\mu (\text{div}^k_\mu (X)X). 
\end{eqnarray}
In particular, identity \eqref{eq:akdiv} holds tautologically for $a_1$ and by Definition~\ref{def: H} for $a_2$. 
For what concerns the coefficient $a_3$, we have that
\begin{equation}
\mathrm{div}^2_\mu (\nabla_H\delta)=\mathrm{div}_\mu (\mathrm{div}_\mu (\nabla_H\delta)\nabla_H\delta)=-\mathrm{div}(\mathcal{H}_\mu N)= -N\left(\mathcal{H}_\mu\right)+\mathcal{H}_\mu^2,
\end{equation}
where $N=\nabla_H\delta$ by Theorem~\ref{thm:smoothness} and thanks to the property of the divergence for which $\mathrm{div}(aX)=X(a)+a\mathrm{div}(X)$ for every smooth function $a$.
 This fact extends the result in \cite{balogh15} for a surface in the Heisenberg group endowed with the Popp measure.
In fact, $\Popp(U_{\varepsilon})$ is analytic in $\eps$
 \begin{equation}
 \Popp(U_{\varepsilon})=\sum_{k=1}^{\infty}\frac{\varepsilon^{k}}{k!}\int_{\partial\Omega}\text{div}^{k-1}\left(\nabla_H\delta\right)d\mathcal{H}^3_{d_{cc}},
 \end{equation}
where $d\mathcal{H}^3_{d_{cc}}$ is the three-dimensional Hausdorff measure of the metric of the Carnot-Carath\'eodory distance $d_{cc}$ on the surface.
Moreover, the authors prove that the iterated divergences are polynomials of certain second order derivatives of the distance function for which it is given a precise recursive formula.


\end{remark}

\section{Surfaces in model spaces}
\label{Applications}

We compute expansion \eqref{eq:expansion} in Theorem~\ref{thm:sviluppo} for a class of surfaces in model sub-Riemannian spaces, with $\mu$ equal to the Popp measure $\Popp$. We exploit Proposition~\ref{prop:a3f}, avoiding the explicit computation of the sub-Riemannian distance from the surface. 

First, we consider the class of rotational surfaces in the Heisenberg group $\mathbb{H}$. 
Afterwards, following the construction presented in \cite[Section~5]{BBCH21}, we take into account two specific surfaces in the model spaces $SU(2)$ and $SL(2)$. We show that for these surfaces the coefficients $a_2$ and $a_3$ in expansion  \eqref{eq:expansion} are always vanishing. 
This generalizes the case of the horizontal plane in $\mathbb{H}$ shown in \cite{balogh15}.

\subsection{Rotational surfaces in $\mathbb{H}$}

Let us consider the Heisenberg group $\mathbb{H}=(\R^3,\omega)$, where the contact form is 
\begin{equation}
\omega=dz-\frac{y}{2}dx+\frac{x}{2}dy, 
\end{equation}
the distribution $\distr=\ker \omega$ is spanned by the vector fields
\begin{eqnarray}
\label{eq:vectorfieldsH}
X_1=\frac{\partial}{\partial x} +\frac{y}{2}\frac{\partial}{\partial z} &\mbox{and}& X_2=\frac{\partial}{\partial y} -\frac{x}{2}\frac{\partial}{\partial z},
\end{eqnarray} 
and the Reeb vector field is $X_0=\frac{\partial}{\partial z}$.
Moreover, let $f:\mathbb{H}\to\R$ be defined by
\begin{equation}
f(x,y,z)=z-g(r)
\end{equation}  
where $r=\sqrt{x^2+y^2}$ and $g:[0,+\infty)\to\R$ is a smooth function such that $g'(0)=0$.
The level set $S=f^{-1}(0)$ is a smooth surface with a unique characteristic point in $(0,0,g(0))$. 
Indeed, if one considers other points of the surface, the vector field $-y\partial_x+x\partial_y$ is tangent to the surface at such points, but it is not contained in the distribution $\distr$, while the distribution coincides with the tangent plane to $S$ at $(0,0,g(0))$ because $\partial_xf=\partial_yf=0$ at such a point.

Let us consider $U\subset S$ an open and relatively compact set such that its closure does not contain the characteristic point $(0,0,g(0))$. 
In order to obtain the coefficient $a_2$, we proceed to compute the mean curvature $\mathcal{H}$ restricted on the surface exploiting \eqref{eq:HU}.
From \eqref{eq:divergenceX} in Lemma~\ref{lem: H}, the divergence with respect to the Popp's measure $\Popp$  of both $X_1$ and $X_2$ in $\mathbb{H}$ is zero, therefore the quantity we need to is the following
\begin{equation}
\label{eq:Hf}
-\mathcal{H}|_S=\frac{X_1X_1f+X_2X_2f}{\|\nabla_Hf\|}+\nabla_Hf\left( \frac{1}{\|\nabla_Hf\|}\right).
\end{equation}
We compute separately the elements composing the previous formula:
\begin{gather}
X_1f=\frac{y}{2}-g'(r)\frac{x}{r},\qquad X_2f=-\frac{x}{2}-g'(r)\frac{y}{r};\\
X_1X_1f+X_2X_2f=-g''(r)-\frac{g'(r)}{r};\\
\frac{1}{\left\|\nabla_Hf\right\|}=\frac{1}{\sqrt{(X_1f)^2+(X_2f)^2}}=\frac{2}{\sqrt{r^2+4g'(r)^2}}.
\end{gather}
Hence we can compute
\begin{gather}
\nabla_Hf=\left(\frac{y}{2}-g'(r)\frac{x}{r}\right)\partial_x-\left(\frac{x}{2}+g'(r)\frac{y}{r}\right)\partial_y+\frac{r^2}{4}\partial_z;\\
\nabla_Hf\left( \frac{1}{\|\nabla_Hf\|}\right)=
-2\frac{r+4g'(r)g''(r)}{(r^2+4g'(r)^2)^\frac{3}{2}}\nabla_Hf(r)=
2g'(r) \frac{r+4g'(r)g''(r)}{(r^2+4g'(r)^2)^\frac{3}{2}}.
\end{gather}
Plugging into \eqref{eq:Hf}, we obtain that
\begin{align}
-\mathcal{H}
&=-2\frac{g''(r)+\frac{g'(r)}{r}}{\sqrt{r^2+4g'(r)^2}} +2g'(r) \frac{r+4g'(r)g''(r)}{(r^2+4g'(r)^2)^\frac{3}{2}}
=-2\frac{4g'(r)^3+r^3g''(r)}{r(r^2+4g'(r)^2)^\frac{3}{2}}.
\end{align}
Finally, we compute the elements needed for $a_3$ exploiting Proposition~\ref{prop:a3f}:
\begin{align}
X_S&=\frac{X_2f}{\|\nabla_Hf\|}X_1-\frac{X_1f}{\|\nabla_Hf\|}X_2\\
&=\frac{1}{\|\nabla_Hf\|}\left[ -\left(\frac{x}{2}+g'(r)\frac{y}{r}\right)\partial_x-\left(\frac{y}{2}-g'(r)\frac{x}{r}\right)\partial_y-\frac{r}{2}g'(r)\partial_z\right]
\end{align}
and
\begin{align}
2X_S\left(\frac{X_0f}{\|\nabla_Hf\|}\right)-\left(\frac{X_0f}{\|\nabla_Hf\|}\right)^2
&=-\frac{2}{\|\nabla_Hf\|^3}\left(r+4g'(r)g''(r)\right)X_S(r)-\frac{1}{\|\nabla_Hf\|^2}\\
&=4g'(r)\frac{rg''(r)-g'(r)}{\|\nabla_Hf\|^4}.
\end{align}
We conclude that 
\begin{align}
\Popp(U_\eps)&=\eps\int_UdA_{}-\eps^2\int_U\frac{4g'(r)^3+r^3g''(r)}{r(r^2+4g'(r)^2)^\frac{3}{2}}dA_{}\\
&\quad+\frac{\eps^3}{6}\int_U 4g'(r)\frac{rg''(r)-g'(r)}{(r^2+4g'(r)^2)^2} dA_{}+o(\eps^3).
\end{align}

\subsection{Examples in model spaces}
Let us consider in the Heisenberg group $\mathbb{H}$ the plane $S=f^{-1}(0)$ with $f(x,y,z)=z$.
The surface $S$ bounds $\Omega=\{z\leq 0\}$ and has an unique characteristic point in the origin. 
Let $U\subset S$ be open and relatively compact such that $(0,0,0)\notin \overline{U}$. Considering the same setting as in the previous subsection, one can compute expansion~\eqref{eq:expansion} of Theorem~\ref{thm:sviluppo} obtaining that
\begin{equation}
\label{eq:volumepiano}
\Popp(U_\eps)=\eps\int_UdA_{}+o(\eps^3).
\end{equation}
This result agrees with the example presented in \cite{balogh15}.
In what follows the spaces of interest are the Lie groups $SU(2)$ and $SL(2)$ equipped with standard sub-Rieman\-nian structures. Together with the Heisenberg group, these spaces play the role of model spaces for the three-dimensional contact sub-Rieman\-nian manifolds.

\subsubsection{The Special Unitary Group $SU(2)$}

The Special Unitary group is the Lie group of the $2\times2$ unitary matrices of determinant $1$ with the 
 matrix multiplication,
\begin{equation}
SU(2)=\left\{
\left(\begin{array}{cc}
z+iw & y+ix\\
-y+ix & z-iw 
\end{array}\right)
: x,y,z,w\in\R \text{ such that } x^2+y^2+z^2+w^2=1\right\}.
\end{equation} 
The Lie algebra $su(2)$ is the algebra of the $2\times 2$ skew-Hermitian matrices with vanishing trace. Let us identify $SU(2)$ with the unitary sphere $S^3\subset\R^4$, and given $k>0$, we define the left-invariant vector fields on the Lie group:
\begin{gather}
X_1=k\left(z\frac{\partial}{\partial x}-w\frac{\partial}{\partial y}-x\frac{\partial}{\partial z}+y\frac{\partial}{\partial w}\right),\\ 
X_2=k\left(w\frac{\partial}{\partial x}+z\frac{\partial}{\partial y}-y\frac{\partial}{\partial z}-x\frac{\partial}{\partial w}\right),\\
X_0=2k^2\left(-y\frac{\partial}{\partial x}+x\frac{\partial}{\partial y}-w\frac{\partial}{\partial z}+z\frac{\partial}{\partial w}\right).
\end{gather}
The relations given by the Lie brackets between these vector fields correspond to that in \eqref{eq:bracketsk} with $\kappa=4k^2$. 
Now, let us consider the smooth function $f(x,y,z,w)=w$ that defines the surface $S=\{w=0\}\simeq S^2$ . We have that
\begin{eqnarray}
X_1f=ky, & X_2f=-kx, &X_0f=2k^2z.
\end{eqnarray}
There are two characteristic points $(0,0,\pm 1,0)$, that are two opposite poles in the sphere, and correspond to $\pm \mathbb{I} $, where $\mathbb{I}$ is the identity $2\times 2$ matrix.
Let us choose $U\subset S$ relatively compact such that the closure does not contain any characteristic point. In coordinates,
\begin{equation}
\left\|\nabla_Hf\right\|=k\sqrt{x^2+y^2}, 
\qquad \nabla_Hf|_U=k^2yz\dfrac{\partial}{\partial x}-k^2xz\frac{\partial}{\partial y}+k^2(x^2+y^2)\dfrac{\partial}{\partial w}.
\end{equation}
From \eqref{eq:divergenceX} in Lemma~\ref{lem: H}, we have that $\text{div}X_1=\text{div}X_2=0$, and so we compute the mean sub-Riemannian curvature as previously with \eqref{eq:Hf}.
\begin{equation}
-\mathcal{H}|_U
=\dfrac{X_1(y)-X_2(x)}{\sqrt{x^2+y^2}}-\dfrac{x\nabla_Hf(x)+y\nabla_Hf(y)}{k(x^2+y^2)^{\frac{3}{2}}}=0.
\end{equation}
For $a_3$ we exploit Proposition~\ref{prop:a3f} (more precisely \eqref{eq:corollary} in Corollary~\ref{cor:corollary}):
\begin{gather}
\dfrac{X_0f}{\left\|\nabla_Hf\right\|}=\dfrac{2kz}{\sqrt{x^2+y^2}}, \qquad X_S=\dfrac{k}{\sqrt{x^2+y^2}}\left(-xz\frac{\partial}{\partial x}-yz\frac{\partial}{\partial y}+(x^2+y^2)\frac{\partial}{\partial z}\right).
\end{gather}
And thus, we have the following formula, which proves \eqref{eq:volumepiano}
\begin{align}
a_3&=\int_U\left[\dfrac{4k}{\sqrt{x^2+y^2}}X_S(z)-\dfrac{4kz}{(x^2+y^2)^\frac32}(xX_S(x)+yX_S(y))-\dfrac{4k^2z^2}{x^2+y^2}-4k^2\right]dA_{}\\
&=\int_U\left[4k^2+\dfrac{4k^2z^2}{x^2+y^2}-\dfrac{4k^2z^2}{x^2+y^2}-4k^2\right]dA_{}=0.
\end{align}

\subsubsection{The Special Linear Group $SL(2,\R)$}

The special linear group $SL(2,\R)$  is the Lie group of $2\times2$ matrices with determinant $1$  with the matrix multiplication,
\begin{equation}
SL(2,\R)=\left\{(x,y,z,w)\in\R^4 : xw-yz=1\right\}.
\end{equation}
The corresponding Lie algebra $sl(2,\R)$ is the vector space of the $2\times2$ real matrices with vanishing trace.
For $k>0$, we define the left invariant vector fields
\begin{gather}
X_1=k\left(y\frac{\partial}{\partial x}+x\frac{\partial}{\partial y}+w\frac{\partial}{\partial z}+z\frac{\partial}{\partial w}\right),\\
X_2=k\left(x\frac{\partial}{\partial x}-y\frac{\partial}{\partial y}+z\frac{\partial}{\partial z}-w\frac{\partial}{\partial w}\right),\\
X_0=2k^2\left(-y\frac{\partial}{\partial x}+x\frac{\partial}{\partial y}-w\frac{\partial}{\partial z}+z\frac{\partial}{\partial w}\right).
\end{gather}
The commutation relations for these vector fields, are of type in \eqref{eq:Hf} with $\kappa=-4k^2$. 
Let us consider $f(x,y,z,w)=y-z$ defining the plane $S=f^{-1}(0)$. We obtain that
\begin{equation}
X_1f=k(x-w), \qquad X_2f=-k(y+z), \qquad X_0f=2k^2(x+w).
\end{equation}
The characteristic points are $(\pm 1,0,0,\pm 1)$ that correspond to $\pm \mathbb{I} $, where $\mathbb{I}$ is the identity $2\times 2$ matrix.
Let us choose $U\subset S$ relatively compact such that the closure does not contain any characteristic point.
 We have that 
\begin{equation}
\left\|\nabla_Hf\right\|=k\sqrt{(x-w)^2+(y+z)^2}
\end{equation}
and, considering that on $S$ it holds the relation $xw=1+y^2$,
\begin{equation}
\nabla_Hf_{\mid S}=k^2y(w+x)\left(\dfrac{\partial}{\partial w}-\dfrac{\partial}{\partial x}\right)+k^2(x^2+y^2-1)\frac{\partial}{\partial y}+k^2(1-y^2-w^2)\dfrac{\partial}{\partial z}.
\end{equation}
From \eqref{eq:divergenceX} in Lemma~\ref{lem: H}, we have that $\text{div}X_1=\text{div}X_2=0$, and so we compute the mean sub-Riemannian curvature as previously with \eqref{eq:Hf}.
\begin{align}
-\mathcal{H}|_U
&=\dfrac{X_1(x-w)-X_2(y+z)}{\sqrt{(x-w)^2+(y+z)^2}}-\dfrac{(x-w)\nabla_Hf(x-w)+(y+z)\nabla_Hf(y+z)}{k((x-w)^2+(y+z)^2)^{\frac{3}{2}}}\\
&=2k\frac{y-z}{\sqrt{(x-w)^2+(y+z)^2}}-k\dfrac{-(x-w)2y(x+w)+(y+z)(x^2-w^2)}{(x^2+2y^2+w^2-2)^{\frac{3}{2}}}=0.
\end{align}
For the coefficient $a_3$, we compute the characteristic vector field $X_S$ on $U$
\begin{equation}
X_S|_U=\frac{-k^2}{\|\nabla_Hf\|}\left[(x^2+y^2-1)\frac{\partial}{\partial x}+y(w+x)\left(\frac{\partial}{\partial y}+\frac{\partial}{\partial z}\right)+(y^2+w^2-1)\frac{\partial}{\partial w}\right].
\end{equation}
And thus, from Proposition~\ref{prop:a3f} (more precisely, from \eqref{eq:corollary} in Corollary~\ref{cor:corollary}),
\begin{align}
a_3&=\int_U\left[2X_S\left(\frac{2k(x+w)}{\sqrt{(x-w)^2+(y+z)^2}}\right)-\left(\frac{2k(x+w)}{\sqrt{(x-w)^2+(y+z)^2}}\right)^2+4k^2\right]dA_{}\\
&=\int_U\left[\dfrac{4kX_S(x+w)}{\sqrt{(x-w)^2+(y+z)^2}}-4k(x+w)\dfrac{(x-w)X_S(x-w)+2yX_S(y+z)}{((x-w)^2+(y+z)^2)^\frac32}\right.\\
&\qquad\quad \left. -\frac{4k^2(x+w)^2}{(x-w)^2+(y+z)^2}+4k^2\right]dA_{}\\
&=\int_U\left[-4k^2\frac{x^2+2y^2+w^2-2}{(x-w)^2+(y+z)^2}-4k^2(x+w)\dfrac{(x-w)(x^2-w^2)+4y^2(x+w)}{((x-w)^2+(y+z)^2)^2}\right.\\
&\qquad\quad \left. +\frac{4k^2(x+w)^2}{(x-w)^2+(y+z)^2}+4k^2\right]dA_{}=0
\end{align}
after some simplifications. We conclude, again, that \eqref{eq:volumepiano} holds.

\appendix
\section{Proof of Proposition~\ref{prop:Hriemannapproximation}} \label{a:mean}

We denote by $\langle \cdot  , \cdot \rangle_\eps$ the scalar product induced by $g^\eps$.
Moreover, let $\nabla^\eps$ be the Levi-Civita connection associated to $g^\eps$. 
In order to compute the mean curvature $H^\eps$ of $S$ with respect to the $g^\eps$, we consider the orthonormal frame on $TS$ 
\begin{equation}
X_S=(X_2\delta)X_1-(X_1\delta)X_2 \qquad Y^\eps=\frac{\eps(X_\theta\delta)(X_1\delta)X_1+\eps(X_\theta\delta)(X_2\delta)X_2-\eps X_\theta}{\sqrt{1+\eps^2(X_\theta\delta)^2}}.
\end{equation} 
Notice that $X_S$ is the characteristic vector field in \eqref{eq:XS} which is horizontal and does not depend on $\eps$. Moreover, if $X_\theta$ is the Reeb vector field $X_0$, then $Y^\eps$ is $Y_S$ in \eqref{eq: YS} normalized with respect to $g^\eps$.
The mean curvature $H^\eps$ is the trace of the second fundamental form computed with respect to the frame $X_S,Y$. More precisely,
\begin{equation}
\label{eq:He}
H^\eps = \langle\nabla^\eps_{X_S}X_S,N^\eps\rangle_\eps+ \langle\nabla^\eps_{Y^\eps}Y^\eps,N^\eps\rangle_\eps
\end{equation}  
where $N^\eps$ is the Riemannian gradient with respect to $g^\eps$ of $S$
\begin{equation}
N^\eps =\frac{(X_1\delta)X_1+(X_2\delta)X_2+\eps^2(X_\theta\delta)X_\theta}{\sqrt{1+\eps^2(X_\theta\delta)^2}}.
\end{equation}
The key tool for the computations is the Koszul formula, i.e.,
for $U,V,Z$ vector fields, we have that 
\begin{equation}
\label{eq:Koszul}
\langle\nabla^\eps_UV,Z\rangle_\eps=\frac{1}{2}\left(\langle[U,V],Z\rangle_\eps-\langle[V,Z],U\rangle+\langle[Z,U],V\rangle_\eps\right).
\end{equation}
Furthermore recall that 
\begin{equation}
\label{eq:ckij}
[X_j,X_i]=c^1_{ij}X_1+c^2_{ij}X_2+c^\theta_{ij}X_\theta \quad \mbox{for } i,j=1,2,\theta.
\end{equation}
Notice that $c^k_{ij}$ do not depend on $\eps$.
Let us compute the first term in \eqref{eq:He}
\begin{align}
\nabla^\eps_{X_S}X_S&=\nabla^\eps_{(X_2\delta)X_1-(X_1\delta)X_2}(X_2\delta)X_1-(X_1\delta)X_2 \\
& = \left((X_2\delta)X_1X_2\delta-(X_1\delta)X_2X_2\delta\right)X_1\\
& \quad +\left(-(X_2\delta)X_1X_1\delta+(X_1\delta)X_2X_1\delta\right)X_2\\
& \quad + (X_2\delta)^2 \nabla^\eps_{X_1}X_1+(X_1\delta)^2\nabla^\eps_{X_2}X_2\\
& \quad -(X_1\delta)(X_2\delta)(\nabla^\eps_{X_1}X_2+\nabla^\eps_{X_2}X_1).
\end{align}
Differentiating \eqref{eq:gradH=1} with respect to $X_1$ and $X_2$ and $X_\theta$ we obtain that
\begin{eqnarray}
\label{eq:derivazionegradienteH=1}
(X_1\delta)X_1X_1\delta+(X_2\delta)X_1X_2\delta=0,\\
(X_1\delta)X_2X_1\delta+(X_2\delta)X_2X_2\delta=0,\\
(X_1\delta)X_\theta X_1\delta+(X_2\delta)X_\theta X_2\delta=0.
\end{eqnarray}
Taking into account \eqref{eq:ckij}, we have that 
\begin{align}
\nabla^\eps_{X_1}X_1&=c^1_{12}X_2-\eps^2c^1_{\theta1}X_\theta;\\
\nabla^\eps_{X_1}X_2&=-c^1_{12}X_1+\frac{1}{2}(-c^\theta_{12}-\eps^2c^1_{\theta2}-\eps^2c^2_{\theta1})X_\theta\\
\nabla^\eps_{X_2}X_1&=c^2_{12}X_2+\frac{1}{2}(c^\theta_{12}-\eps^2c^1_{\theta2}-\eps^2c^2_{\theta1})X_\theta\\
\nabla^\eps_{X_2}X_2&=-c^2_{12}X_2-\eps^2c^2_{\theta2}X_\theta.
\end{align}
Therefore,
\begin{align}
\nabla^{\eps}_{X_S}X_S&=\left(-(X_1\delta)X_1X_1\delta-(X_1\delta)X_2X_2\delta\right)X_1\\
& \quad +\left(-(X_2\delta)X_1X_1\delta-(X_2\delta)X_2X_2\delta\right)X_2\\
& \quad + (X_2\delta)^2 (c^1_{12}X_2-\eps^2c^1_{\theta1}X_\theta)+(X_1\delta)^2(-c^2_{12}X_2-\eps^2c^2_{\theta2}X_\theta)\\
& \quad -(X_1\delta)(X_2\delta)\left(-c^1_{12}X_1+c^2_{12}X_2+-(c^1_{\theta2}+c^2_{\theta1})X_\theta\right)\\
&=-\left(X_1X_1\delta+X_2X_2\delta-c^1_{12}(X_2\delta)+c^2_{12}(X_1\delta)\right)\left((X_1\delta)X_1+(X_2\delta)X_2\right)\\
& \quad -\eps^2((X_2\delta)^2c^1_{\theta1}+(X_1\delta)^2c^2_{\theta2}-(X_1\delta)(X_2\delta)(c^1_{\theta2}+c^2_{\theta1}))X_\theta.
\end{align}
Simplifying  the computations we obtain that
\begin{align}
\langle\nabla^\eps_{X_S}X_S,N^\eps\rangle_\eps&=\frac{-X_1X_1\delta-X_2X_2\delta+c^1_{12}(X_2\delta)-c^2_{12}(X_1\delta)}{\sqrt{1+\eps^2(X_\theta\delta)^2}}\\
& \quad -\eps^2 (X_\theta\delta) \frac{(X_2\delta)^2c^1_{\theta1}+(X_1\delta)^2c^2_{\theta2}-(X_1\delta)(X_2\delta)(c^1_{\theta2}+c^2_{\theta1}) }{\sqrt{1+\eps^2(X_\theta\delta)^2}}.
\end{align}
Finally, recalling \eqref{eq: H}, we conclude that 
\begin{equation}
\label{eq:Heps1}
\langle\nabla^\eps_{X_S}X_S,N^\eps\rangle_\eps=\frac{\mathcal{H}}{\sqrt{1+\eps^2(X_\theta\delta)^2}}+O(\eps^2).
\end{equation}
Next, we compute the second term in \eqref{eq:He}.
Let us set $N=(X_1\delta)X_1+(X_2\delta)X_2$ as in Definition~\ref{def: moving frame}, so that we can rewrite $Y^\eps$ and $N^\eps$ in the following way
\begin{equation}
Y^\eps =\frac{\eps(X_\theta\delta)N-\eps X_\theta}{\sqrt{1+\eps^2(X_\theta\delta)^2}}, 
\qquad
N^\eps =\frac{N+\eps^2(X_\theta\delta)X_\theta}{\sqrt{1+\eps^2(X_\theta\delta)^2}}.
\end{equation}
Since $Y^\eps$ and $N^\eps$ are orthogonal,
\begin{equation}
\langle\nabla^\eps_{Y^\eps}Y^\eps,N^\eps\rangle_\eps
=\frac{\eps^2}{(1+\eps^2(X_\theta\delta)^2)^\frac{3}{2}}\langle\nabla^\eps_{(X_\theta\delta)N- X_\theta}((X_\theta\delta)N- X_\theta),N+\eps^2(X_\theta\delta)X_\theta\rangle_\eps.
\end{equation}
In particular,
\begin{align}
\nabla^\eps_{(X_\theta\delta)N- X_\theta}((X_\theta\delta)N- X_\theta) &= (X_\theta\delta)^2\nabla^\eps_NN-(X_\theta\delta)\nabla^\eps_NX_\theta+\nabla^\eps_{X_\theta}X_\theta \\
& \quad +((X_\theta\delta)NX_\theta\delta-X_\theta X_\theta\delta)N.
\end{align}
Since $\nabla^\eps$ is a metric connection, then $\nabla^\eps_UV$ is orthogonal to $V$ for every pair of vector fields $U,V$. Hence 
\begin{align}
\langle\nabla^\eps_{Y^\eps}Y^\eps,N^\eps\rangle_\eps&=-\eps^2\frac{\eps^2(X_\theta\delta)^3\langle_\eps\nabla^\eps_NN,X_\theta\rangle_\eps-\langle(X_\theta\delta)\nabla^\eps_NX_\theta-\nabla^\eps_{X_\theta}X_\theta,N\rangle_\eps  }{(1+\eps^2(X_\theta\delta)^2)^\frac{3}{2}}  \label{eq:privius3} \\
&\quad - \eps^2\frac{(X_\theta\delta)NX_\theta\delta-X_\theta X_\theta\delta }{(1+\eps^2(X_\theta\delta)^2)^\frac{3}{2}}.
\end{align}
Using again the Koszul formula in \eqref{eq:Koszul}, and \eqref{eq:derivazionegradienteH=1}, we obtain that
\begin{align} 
\langle\nabla^\eps_NN,X_\theta\rangle_\eps
&=-\langle\nabla^\eps_NX_\theta,N\rangle _\eps \label{eq:privius2}\\
&=-\langle[N,X_\theta],N\rangle_\eps   \\
&=-(X_1\delta)^2c^1_{\theta1}-(X_1\delta)(X_2\delta)(c^1_{\theta2}+c^2_{\theta1})-(X_2\delta)^2c^2_{\theta2},
\end{align}
that does not depend on $\eps$. On the other hand,
\begin{equation}\label{eq:privius1}
\langle\nabla^\eps_{X_\theta}X_\theta,N\rangle_\eps=\langle[N,X_\theta],X_\theta\rangle_\eps=\frac{(X_1\delta)c^\theta_{\theta1}+(X_2\delta)c^\theta_{\theta2}}{\eps^2}.
\end{equation}
Substituting \eqref{eq:privius2} and \eqref{eq:privius1} into \eqref{eq:privius3}, we have 
\begin{equation}
\label{eq:Heps2}
\langle\nabla^\eps_{Y^\eps}Y^\eps,N^\eps\rangle_\eps=O(\eps^2)
-\frac{(X_1\delta)c^\theta_{\theta1}+(X_2\delta)c^\theta_{\theta2}}{(1+\eps^2(X_\theta\delta)^2)^\frac{3}{2}}
\end{equation}
Taking the limit for $\eps\to 0$ in \eqref{eq:Heps1} and \eqref{eq:Heps2}
\begin{equation}
\lim_{\eps\to 0}H^\eps=\lim_{\eps\to 0}\frac{\mathcal{H}+(X_1\delta)c^\theta_{\theta1}+(X_2\delta)c^\theta_{\theta2}}{(1+\eps^2(X_\theta\delta)^2)^\frac{3}{2}}+O(\eps^2)
\end{equation}
one recovers \eqref{eq:Hmu2}.

\bibliography{0-SR-Tube-VF}

\end{document}